\documentclass{emsart}

\usepackage[]{graphicx}
\usepackage{psfrag}

\contact[Thierry.Barbot@umpa.ens-lyon.fr]
{1. Thierry Barbot, CNRS, UMR 5669, 
\'Ecole Normale Sup\'erieure de Lyon, 
46 all\'ee d'Italie, 69364 Lyon, France}

\contact[v.charette@usherbrooke.ca]
{2. Virginie Charette,
D\'epartement de math\'ematiques, Universit\'e de Sherbrooke,
Sherbrooke, Quebec, Canada}

\contact[tad@math.upenn.edu]
{3. Todd A. Drumm
Department of Mathematics, University of Pennsylvania,
Philadelphia, PA USA}

\contact[wmg@math.umd.edu]{4. William M.\ Goldman,
Department of Mathematics, University of Maryland,
College Park, MD 20742 USA}

\contact[
karin.melnick@yale.edu]
{5. Karin Melnick, 
Department of Mathematics,
Yale University, New Haven, CT
06520 USA}

\newtheorem{theorem}{Theorem}[subsection]
\newtheorem{corollary}[theorem]{Corollary}

\newtheorem{proposition}[theorem]{Proposition}

\newtheorem{fact}[theorem]{Fact}

\theoremstyle{definition}
\newtheorem{definition}[theorem]{Definition}
\newtheorem{remark}[theorem]{Remark}

\newcommand{\Oh}{\operatorname{O}}
\newcommand{\SOh}{\operatorname{SO}}
\newcommand{\POh}{\operatorname{PO}}
\newcommand{\GL}{\operatorname{GL}}
\newcommand{\SL}{\operatorname{SL}}
\newcommand{\oh}{\mathfrak{o}}
\newcommand{\ott}{\oh(3,2)}
\newcommand{\RR}{\mathbb{R}}
\newcommand{\CC}{\mathbb{C}}
\newcommand{\EE}{\mathbb{E}}
\newcommand{\ZZ}{\mathbb{Z}}
\newcommand{\RP}{\mathbb{RP}}
\renewcommand{\Pr}{\mathbb{P}}

\newcommand{\Ein}{\operatorname{Ein}}
\newcommand{\Rnt}{\RR^{n+1,2}}
\newcommand{\Rno}{\RR^{n,1}}
\newcommand{\Rtt}{\RR^{3,2}}
\newcommand{\NN}{{\mathfrak N}}
\newcommand{\Nnt}{\NN^{n+1,2}}
\newcommand{\Nno}{\NN^{n+1,1}}
\newcommand{\Ntt}{\NN^{3,2}}
\newcommand{\Einno}{\Ein^{n,1}}
\newcommand{\EinTO}{\Ein^{2,1}}
\newcommand{\EinOO}{\Ein^{1,1}}
\newcommand{\uEin}{\widehat{\Ein}}
\newcommand{\uEinno}{\uEin^{n,1}}
\newcommand{\uuEin}{\widetilde{\Ein}}
\newcommand{\uuEinno}{\uuEin^{n,1}}
\newcommand{\uuEinTO}{\uuEin^{2,1}}
\newcommand{\uEinTO}{\uEin^{2,1}} 

\newcommand{\Eno}{\EE^{n,1}}

\newcommand{\Pho}{\operatorname{Pho}^{2,1}}
\newcommand{\Min}{\operatorname{Min}}
\newcommand{\dS}{\operatorname{dS}}
\newcommand{\AdS}{\operatorname{AdS}}

\newcommand{\Flag}{\operatorname{Flag}^{2,1}}

\newcommand{\PGL}{\operatorname{PGL}}
\newcommand{\Sp}{\operatorname{Sp}}
\renewcommand{\sp}{\mathfrak{sp}}
\renewcommand{\sl}{\mathfrak{sl}}
\newcommand{\spf}{\sp(4,\RR)}
\newcommand{\vol}{\operatorname{vol}}
\newcommand{\BB}{\mathbb{B}}
\newcommand{\so}{\mathfrak{so}}
\newcommand{\Eto}{\EE^{2,1}}
\newcommand{\Rto}{\RR^{2,1}}

\newcommand{\Hom}{\operatorname{Hom}}

\newcommand{\Sim}{\operatorname{Sim}}
\newcommand{\Fix}{\operatorname{Fix}}

\newcommand{\aff}{\operatorname{aff}}

\newcommand{\Id}{\operatorname{I}} 
\newcommand{\II}{\operatorname{I}}  

\newcommand{\JJ}{\operatorname{J}}
\newcommand{\UU}{\operatorname{U}}

\newcommand{\Lag}{\operatorname{Lag}}
\newcommand{\Mas}{\operatorname{Maslov}}
\newcommand{\Mat}{\operatorname{Mat}_2(\RR)}

\newcommand{\Ww}{\mathcal{W}}
\newcommand{\Ee}{\mathcal{E}}
\newcommand{\Ss}{\mathcal{S}}

\renewcommand{\aa}{\mathfrak{a}}  
\newcommand{\hh}{\mathfrak{h}}

\renewcommand{\ggg}{\mathfrak{g}}

\newcommand{\ppp}{\mathfrak{p}}
\newcommand{\uu}{\mathfrak{u}}

\usepackage{makeidx}
\makeindex

\title[The $(2+1)$-Einstein universe]
{A primer on the $(2+1)$ Einstein universe}

\author[ Barbot, Charette, Drumm, Goldman, and Melnick ]{
Thierry Barbot, Virginie Charette, Todd Drumm,
William M. Goldman, and Karin Melnick
\thanks{Goldman is grateful to 
the Erwin Schrodinger Institute (Vienna) for hospitality during the writing of
this paper, and to National Science Foundation  grant 
DMS-0405605. Charette is grateful for support from NSERC.}}

\begin{document}

\begin{abstract}
The Einstein universe is the conformal compactification of Minkowski
space. It also arises as the ideal boundary of anti-de Sitter space.
The purpose of this article is to develop the synthetic geometry of
the Einstein universe in terms of its homogeneous submanifolds and
causal structure, with particular emphasis on dimension $2 + 1$, in
which there is a rich interplay with symplectic geometry.
\end{abstract}



\maketitle
\tableofcontents
\section{Introduction}

We will explore the geometry of the conformal compactification of
Minkowski $(n+1)$--space inside of $\RR^{n,2}$. We shall call this
conformal compactification $\Einno$, or the {\em Einstein universe},
and its universal cover will be denoted $\widetilde{\Ein}^{n,1}$.  The
Einstein universe is a homogeneous space $G/P$, where
$G=\operatorname{PO}(n,2)$, and $P$ is a parabolic subgroup.  When
$n=3$, then $G$ is locally isomorphic to $\operatorname{Sp}(4, \RR)$.

The origin of the terminology ``Einstein universe'' is that A.\
Einstein himself considered as a paradigmatic universe the product
$S^{3} \times \RR$ endowed with the Lorentz metric $ds_{0}^{2} -
dt^{2}$, where $ds_0^2$ is the usual constant curvature Riemannian
metric on $S^3$.  The conformal transformations preserve the class of
lightlike geodesics and provide a more flexible geometry than that
given by the metric tensor.

Our motivation is to understand conformally flat
Lorentz manifolds and the Lorentzian analog of Kleinian groups. Such
manifolds are locally homogeneous geometric structures modeled on $\EinTO$.
  
The Einstein universe $\Einno$ is the {\em conformal compactification\/} of Minkowski space $\Eno$ in the
same sense that the $n$-sphere
\begin{equation*}
S^n = \EE^n \cup \{\infty\} 
\end{equation*}
conformally compactifies {\em Euclidean space $\EE^n$\/}; in particular, a Lorentzian analog of the following theorem holds (see \cite{Frances}): 

\begin{theorem}[Liouville's theorem] Suppose $n\geq 3$. Then every
conformal map $U\xrightarrow{f}\EE^n$ defined on a nonempty connected
subdomain $U \subset\EE^n$ extends to a conformal automorphism $\bar{f}$
of $S^n$. Furthermore $\bar{f}$ lies in the group $\POh(n+1,1)$
generated by inversions in hyperspheres and Euclidean isometries.
\end{theorem}

Our viewpoint involves various geometric objects in Einstein
space: points are organized into $1$-dimensional submanifolds which we
call {\em photons,\/} as they are lightlike geodesics. Photons in turn
form various subvarieties, such as lightcones and hyperspheres.  For
example, a lightcone is the union of all photons through a given
point.  Hyperspheres fall into two types, depending on the signature
of the induced conformal metric. {\em Einstein hyperspheres\/} are
Lorentzian, and are models of $\Ein^{n-1,1}$, while {\em spacelike
hyperspheres\/} are models of $S^n$ with conformal Euclidean geometry.

The Einstein universe $\Einno$ can be constructed by projectivizing
the nullcone in the inner product space $\RR^{n+1,2}$ defined by a
symmetric bilinear form of type $(n+1,2)$.  Thus the points of
$\Einno$ are {\em null lines\/} in $\RR^{n+1,2}$, and photons correspond
to isotropic $2$-planes. Linear hyperplanes $H$ in $\RR^{n+1,2}$
determine lightcones, Einstein hyperspheres, and spacelike
hyperspheres, respectively, depending on whether the restriction of
the bilinear form to $H$ is degenerate, type $(n,2)$, or
Lorentzian, respectively.

Section~\ref{sec.causal} discusses {\em causality\/} in Einstein space.  
Section~\ref{sec.symplectic} is specific to dimension 3,
where the conformal Lorentz group is locally isomorphic to the group
of linear symplectomorphisms of $\RR^4$. This establishes a close
relationship between the symplectic geometry of $\RR^4$ (and hence the
contact geometry of $\RP^3$) and the conformal Lorentzian geometry of $\EinTO$.  
Section~\ref{sec.liealgebra} reinterprets these synthetic
geometries in terms of the structure theory of Lie algebras.  
Section~\ref{sec.dynamic} discusses the dynamical theory of discrete subgroups
of $\EinTO$ due to Frances~\cite{Frances2}, 
and begun by Kulkarni~\cite{Kulkarni}.  
Section~\ref{sec.crooked} discusses the
{\em crooked planes\/}, discovered by Drumm~\cite{drumm1}, in the
context of $\EinTO$; their closures, called {\em crooked surfaces\/}
are studied and shown to be Klein bottles invariant under the
Cartan subgroup of $\SOh(3,2)$.  The paper concludes with
a brief description of discrete groups of conformal transformations
and some open questions.

Much of this work was motivated by the thesis of Charles
Frances~\cite{Frances}, which contains many constructions and
examples, his paper~\cite{Frances2} on Lorentzian Kleinian groups, and
his note~\cite{Frances} on compactifying crooked planes.  We are
grateful to Charles Frances and Anna Wienhard for many useful
discussions.

We are also grateful to the many institutions where we have been able
to meet to discuss the mathematics in this paper. In particular, we
are grateful for the hospitality provided by the Banff International
Research Station~\cite{BIRS} where all of us were able to meet for a 
workshop in November 2004, the workshop in Oostende, Belgium in May 2005 on
``Discrete groups and geometric structures,'' the miniconference in
Lorentzian geometry at the E.N.S.\ Lyon in July 2005, the special
semester at the Newton Institute in Cambridge in Fall 2005, the
special semester at the Erwin Schr\"odinger Institute in Fall 2005,
and a seminar at the University of Maryland in summer 2006, when
the writing began.

\section{Synthetic geometry of $\Einno$}

In this section we develop the basic synthetic geometry of {\em
Einstein space,\/} or the {\em Einstein universe}, starting with the
geometry of {\em Minkowski space\/} $\Eno$.

\subsection{Lorentzian vector spaces}
\label{sub.lorentzien}
\index{Lorentzian vector space}
We consider {\em real inner product spaces,\/} that is, vector spaces $V$
over $\RR$ with a nondegenerate symmetric bilinear form $\langle,\rangle$.
A nonsingular symmetric $n\times n$-matrix $B$ 
defines a symmetric bilinear form on $\RR^n$ by the rule:
\begin{equation*}
\langle u,v \rangle_B := u^\dag B v.
\end{equation*}
where $u^\dag$ denotes the transpose of the vector $u$.
We shall denote by $\RR^{p,q}$ a real inner product space whose inner product is of type $(p,q)$. For example, if
\begin{equation*}
u = \bmatrix u_1 \\ \vdots \\ u_p \\ u_{p+1} \\ \vdots \\  u_{p+q} \endbmatrix,
v = \bmatrix v_1 \\ \vdots \\ v_p \\ v_{p+1} \\ \vdots \\  v_{p+q} \endbmatrix,
\end{equation*}
then
\begin{equation*}
\langle u,v \rangle :=
u_1 v_1 + \dots + u_p v_p - u_{p+1} v_{p+1} -  \dots - u_{p+q} v_{p+q}
\end{equation*}
defines a type $(p,q)$ inner product, induced by the matrix $\Id_p \oplus - \Id_q$ on $\RR^{p+q}$. 
The group of linear automorphisms of $\RR^{p,q}$ is $\Oh(p,q)$.

If $B$ is positive definite---that is, $q=0$---then we say that the
inner product space $(V,\langle,\rangle)$ is {\em Euclidean.\/}
If $q=1$, then $(V,\langle,\rangle)$ is {\em Lorentzian.\/}
We may omit reference to the bilinear form if it is clear from context.

If $V$ is Lorentzian, and $v\in V$, then $v$ is:
\begin{itemize}
\item {\em timelike\/}  if $\langle v,v\rangle < 0$;\index{timelike!vector}
\item {\em lightlike\/} (or {\em null\/} or {\em isotropic\/})
if $\langle v,v\rangle = 0$;\index{lightlike!vector}
\item {\em causal\/} if $\langle v,v\rangle \leq 0$;\index{causal!vector}
\item {\em spacelike\/}  if $\langle v,v\rangle > 0$.\index{spacelike!vector}
\end{itemize}
The {\em nullcone\/} $\NN(V)$ in $V$ consists of all null vectors.
\index{nullcone}

If $W\subset V$, then define its {\em orthogonal complement:\/}
\begin{equation*}
W^\perp := \{ v \in V \mid \langle v,w\rangle = 0 \ \forall \ w \in W \}.
\end{equation*}
The hyperplane $v^\perp$ is {\em null\/} (respectively, {\em timelike\/},
{\em spacelike\/}) if $v$ is null (respectively spacelike, timelike).

In the sequel, according to the object of study, we will consider
several symmetric $n\times n$-matrices and the associated type $(p,q)$
symmetric bilinear forms. For different bilinear forms, different
subgroups of $\Oh(p,q)$ are more apparent. For example:

\begin{itemize}
\item 
Using the diagonal matrix 
\begin{equation*}
\Id_p \oplus - \Id_q
\end{equation*}
invariance under the maximal compact subgroup
\begin{equation*}
\Oh(p)\times \Oh(q)  \subset \Oh(p,q)
\end{equation*}
is more apparent.

\item 
Under the bilinear form defined by the matrix
\begin{equation*}
\Id_{p-q} \oplus \bigoplus^q -1/2 \cdot \bmatrix 0 & 1 \\ 1 & 0 \endbmatrix 
\end{equation*}
(if $p\ge q$),
invariance under the {\em Cartan subgroup\/}
\begin{equation*}
\{\Id_{p-q}\} \times \prod^{q} \Oh(1,1)
\end{equation*}
is more apparent.

\item 
Another bilinear form which we use in the last two sections is:
\begin{equation*}
\Id_{p-1} \oplus - \Id_{q-1} \oplus -1/2 \cdot \bmatrix 0 & 1 \\ 1 & 0 \endbmatrix 
\end{equation*}
which is useful in extending subgroups of $\Oh(p-1,q-1)$ to $\Oh(p,q)$.
\end{itemize}

\subsection{Minkowski space}
\label{sub.Minkowski space}
\index{Minkowski space}
{\em Euclidean space\/} $\EE^n$ is the model space for {\em Euclidean
geometry,\/} and can be characterized up to isometry as a simply connected,
geodesically complete, flat Riemannian manifold. 
For us, it will be simpler to describe it as an {\em
affine space\/} whose underlying vector space of translations is a
Euclidean inner product space $\RR^n$. That means $\EE^n$ comes equipped
with a simply transitive vector space of translations
\begin{equation*}
p \mapsto p + v 
\end{equation*}
where $p\in \EE^n$ is a point and $v\in \RR^n$ is a vector
representing a parallel displacement. Under this simply transitive
$\RR^n$-action, each tangent space $T_p(\EE^n)$ naturally identifies
with the vector space $\RR^n$. The Euclidean inner product on $\RR^n$
defines a positive definite symmetric bilinear form on each tangent
space---that is, a {\em Riemannian metric.\/}

{\em Minkowski space\/} $\Eno$ is the Lorentzian analog. It is
characterized up to isometry as a simply connected, geodesically
complete, flat Lorentzian manifold. Equivalently, it is an affine
space whose underlying vector space of translations is $\RR^{n,1}$.

The geodesics in $\Eno$ are paths of the form
\begin{align*}
\RR &\xrightarrow{\gamma} \Eno \\
t & \longmapsto  p_0 + t v
\end{align*}
where $p_0 \in\Eno$ is a point and $v\in\RR^{n,1}$ is a vector.  A path $\gamma$ as above is {\em timelike, lightlike, or spacelike,\/} if the {\em velocity\/} $v$ is timelike, lightlike, or spacelike, respectively.

Let $p\in\Eno$. The {\em affine lightcone\/} $L^{\aff}(p)$ at $p$
is defined as the union of all lightlike geodesics through $p$:
\begin{equation*}
L^{\aff}(p) := \{ p + v \in\Eno \mid \langle v,v\rangle = 0\} .
\end{equation*}
Equivalently $L^{\aff}(p)= p + \NN$ where $\NN\subset\RR^{n,1}$ denotes
the nullcone in $\RR^{n,1}$. The hypersurface $L^{\aff}(p)$ is ruled
by lightlike geodesics; it is singular only at $\{ p \}$.  The
Lorentz form on $\Eno$ restricts to a degenerate metric on
$L^{\aff}(p) \setminus \{ p \}$.
\index{affine lightcone}

A lightlike geodesic $\ell\subset\Eno$ lies in a unique null affine hyperplane.
(We denote this $\ell^\perp$, slightly abusing notation.)
That is, writing $\ell = p + \RR v$, where $v\in\Rno$ is a lightlike vector, the null hyperplane $p + v^\perp$ is independent of the choices
of $p$ and $v$ used to define $\ell$.

The \emph{de Sitter hypersphere} of radius $r$ centered at $p$ is defined as
\begin{equation*}
S_r(p) := \{ p + v \in\Eno \mid \langle v,v\rangle = r^2\} .
\end{equation*}
The Lorentz metric on $\Eno$ restricts to a Lorentz metric on $S_r(p)$
having constant sectional curvature $1/r^2$. It is geodesically
complete and homeomorphic to $S^{n-1}\times\RR$.  It is a model for
{\em de Sitter space\/} $\dS^{n-1,1}$.  

As in Euclidean space, a {\em homothety (centered at $x_0$) \/} is any map conjugate by a translation to scalar multiplication:
\begin{align*}
\Eno & \longrightarrow  \Eno \\
  x       & \longmapsto    x_0 + r (x - x_0) .
\end{align*}
A {\em Minkowski similarity transformation\/} is a composition of
\index{similarity transformation!Minkowski}
an isometry of $\Eno$ with a homothety: 
\begin{equation*}
f: x \longmapsto   r A(x) + b .
\end{equation*}
where $A\in\Oh(n,1), r > 0$ and $b\in \RR^{n,1}$ defines a translation.
Denote the group of similarity transformations of $\Eno$ by
$\Sim(\Eno)$.
\subsection{Einstein space}
\label{sub.defeinstein}

{\em Einstein space\/} $\Einno$ is the projectivized nullcone of $\Rnt$. 
The nullcone is
\begin{equation*}
\Nnt := \{ v\in\Rnt \mid  \langle v,v\rangle = 0 \} 
\end{equation*}
and the {\em $(n+1)$-dimensional Einstein universe\/} 
$\Einno$ is the image
of $\Nnt -\{0\}$ under projectivization:
\begin{equation*}
\Rnt -\{0\} \xrightarrow{\Pr} \RR\Pr^{n+2}.
\end{equation*}
In the sequel, for notational convenience, we will denote $\Pr$ as a map from
$\Rnt$, implicitly assuming that the origin $0$ is removed from any
subset of $\Rnt$ on which we apply $\Pr$.

The {\em double covering\/} $\widehat{\Ein}^{n,1}$ is defined as the
quotient of the nullcone $\Nnt$ by the action by {\em positive scalar
multiplications.\/} For many purposes the double covering may be more
useful than $\Einno$, itself.  We will also consider the universal
covering $\uuEinno$ in \S\ref{sec.causal}.

Writing the bilinear form on $\Rnt$ as $\Id_{n+1} \oplus -\Id_2$, that is,
\begin{equation*}
\langle v,v\rangle = v_1^2 + \dots + v_{n+1}^2 - v_{n+2}^2 - v_{n+3}^2,
\end{equation*}
the nullcone is defined by 
\begin{equation*}
v_1^2 + \dots + v_{n+1}^2 = v_{n+2}^2  + v_{n+3}^2 .
\end{equation*}
This common value is always nonnegative, and if it is zero, then
$v=0$ and $v$ does not correspond to a point in $\Einno$. Dividing by
the positive number $\sqrt{v_{n+2}^2  + v_{n+3}^2}$ we may assume
that 
\begin{equation*}
v_1^2 + \dots + v_{n+1}^2 = v_{n+2}^2  + v_{n+3}^2 = 1
\end{equation*}
which describes the product $S^n\times S^1$. Thus
\begin{equation*}
\uEinno \approx S^n \times S^1.
\end{equation*}
Scalar multiplication by $-1$ acts by the antipodal map on both the
$S^n$ and the $S^1$-factor. On the $S^1$-factor the antipodal map is
a translation of order two, so the quotient
\begin{equation*}
\Einno = \uEinno / \{\pm 1\} 
\end{equation*}
is homeomorphic to the mapping torus of the antipodal map on $S^n$. 
When $n$ is even, 
$\Einno$ is nonorientable and $\uEinno$ is an orientable double covering.
If $n$ is odd, then $\Einno$ is orientable.


The objects in the synthetic geometry of $\Einno$ 
are the following collections of points in $\Einno$:
\begin{itemize}
\item {\em Photons} are projectivizations of totally isotropic $2$-planes. 
We denote the space of photons by $\operatorname{Pho}^{n,1}$.  \index{photon} 
A photon enjoys the natural structure of a {\em real projective line}: each photon $\phi\in\operatorname{Pho}^{n,1}$
admits {\em projective parametrizations}, which are diffeomorphisms of $\phi$ with $\RP^1$ such that if $g$ is an automorphism of $\Einno$ preserving $\phi$, 
then $g|_\phi$ corresponds to a projective transformation of $\RP^1$.
The projective parametrizations are unique up to post-composition with transformations in $\PGL(2,\RR)$.

\item {\em Lightcones} are singular hypersurfaces. 
Given any point $p\in\Einno$, 
the {\em lightcone $L(p)$ with vertex $p$\/} 
is the union of all photons containing $p$:
\begin{equation*}
L(p) := \bigcup \{ \phi\in\operatorname{Pho}^{n,1} \mid p \in \phi \}.
\end{equation*}
The lightcone $L(p)$ can be equivalently defined as 
the projectivization of the orthogonal complement $p^{\perp}\cap \Nnt$.  
The only singular point on $L(p)$ is $p$, 
and $L(p)\setminus \{p\}$ is homeomorphic to $S^{n-1}\times \RR$.
\index{lightcone}
\item
The {\em Minkowski patch\/} $\Min(p)$ determined by an element $p$ of $\Einno$
is the complement of $L(p)$ and has the natural structure of 
Minkowski space $\Eno$, as will be explained in \S\ref{sec.confcomp} below. 
In the double cover, a point $\hat{p}$ determines {\em two\/} Minkowski
patches:
\begin{align*}
\Min^+(\hat{p}) & := 
\{\hat{q}\in \uEinno \mid 
\langle p, q \rangle > 0 \ \forall p,q \in \RR^{n+1,2} 
\text{~representing~} \hat{p}, \hat{q}
\} \\
\Min^-(\hat{p}) & := 
\{\hat{q}\in \uEinno \mid 
\langle p, q \rangle < 0 \ \forall p,q \in \RR^{n+1,2} \text{~representing~} \hat{p}, \hat{q}
\} .
\end{align*}

\index{Minkowski patch}
\index{Minkowski patch!positive}
\index{Minkowski patch!negative}
\item There are two different types of {\em hyperspheres}.
\begin{itemize}

\item {\em Einstein hyperspheres} are closures in $\Einno$ of de
Sitter hyperspheres $S_{r}(p)$ in Minkowski patches as defined in
\S\ref{sub.Minkowski space}. Equivalently, they are projectivizations of $v^{\perp} \cap \Nnt$  
for spacelike vectors $v$.  \index{Einstein!hypersphere}

\item {\em Spacelike hyperspheres} are one-point compactifications of
spacelike hyperplanes like $\RR^n$ in a Minkowski patch $\RR^{n,1}
\subset \Einno$.  Equivalently, they are projectivizations of
$v^{\perp} \cap \Nnt$ for timelike vectors $v$.
\index{spacelike!hypersphere}
\end{itemize}

\item
An {\em anti-de Sitter space\/} $\AdS^{n,1}$ is one component of the
complement of an Einstein hypersphere $\Ein^{n-1,1}\subset\Einno$. It
is homeomorphic to $S^1 \times \RR^n$. Its {\em ideal boundary\/}
is $\Ein^{n-1,1}$. 
\end{itemize}
\index{anti-de Sitter space}

\subsection{$2$-dimensional case}
Because of its special significance, we discuss in detail the geometry
of the $2$-dimensional Einstein universe $\EinOO$.
\begin{itemize}
\item
$\EinOO$ is diffeomorphic to a 2-torus.
\item
Each lightcone $L(p)$ consists of two photons which intersect at $p$.
\item
$\EinOO$ has two foliations $F_-$ and $F_+$ by photons, and the lightcone
$L(p)$ is the union of the leaves through $p$ of the respective foliations.
\item
The leaf space of each foliation naturally identifies with $\RP^1$, and
the mapping
\begin{equation*}
\EinOO \longrightarrow \RP^1 \times \RP^1
\end{equation*}
is equivariant with respect to the isomorphism
\begin{equation*}
\Oh(2,2) \xrightarrow{\cong} \PGL(2,\RR)\times \PGL(2,\RR).
\end{equation*}
\end{itemize}
Here is a useful model (compare Pratoussevitch~\cite{Pratoussevitch}):
The space $\Mat$ of $2\times 2$ real matrices 
with the bilinear form associated to the determinant gives an isomorphism
of inner product spaces:
\begin{align*}
\Mat & \longrightarrow \RR^{2,2} \\
\bmatrix m_{11} & m_{12} \\ m_{21} & m_{22} \endbmatrix 
& \longmapsto 
\bmatrix m_{11} \\ m_{12} \\ m_{21} \\  m_{22} \endbmatrix 
\end{align*}
where $\RR^{2,2}$ is given the bilinear form defined by
\begin{equation*}
\frac12 \bmatrix 
0 & 0 & 0 & 1 \\
0 & 0 & -1 & 0 \\0 & -1 & 0 & 0 \\
1 & 0 & 0 & 0 \endbmatrix .
\end{equation*}
The group $\GL(2,\RR)\times\GL(2,\RR)$ acts on $\Mat$ by:
\begin{equation*}
X \xrightarrow{(A,B)} A X B^{-1} 
\end{equation*}
and induces a local isomorphism
\begin{equation*}
\SL_{\pm}(2,\RR) \times \SL_{\pm}(2,\RR) \longrightarrow \Oh(2,2)
\end{equation*}
where 
\begin{equation*}
\SL_{\pm}(2,\RR) := \{ A\in \GL(2,\RR) \mid \det(A) = \pm 1\}. 
\end{equation*}

Here we will briefly introduce \emph{stems}, which are pieces of
crooked planes, as will be discussed in \S\ref{sec.crooked}
below.  Let $p_0,p_\infty\in\EinOO$ be two points not contained in a
common photon.  Their lightcones intersect in two points $p_1$ and
$p_2$, and the union
\begin{equation*}
L(p_0)\cup L(p_\infty) \subset \EinOO
\end{equation*}
comprises four photons
intersecting in the four points $p_0,p_\infty,p_1,p_2$, such that each
point lies on two photons and each photon contains two of these
points.  This \emph{stem configuration} \index{stem configuration} of
four points and four photons can be represented schematically as in
Figure~\ref{fig.configuration} below.
\begin{figure}[ht]
\begin{center}
\includegraphics[width=3cm, height=3cm]{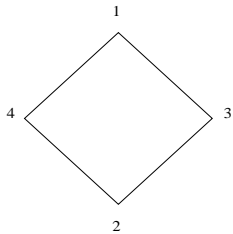}
\end{center}
\caption{Stem Configuration} \label{fig.configuration}
\end{figure}


The complement
\begin{equation*}
\EinOO \setminus \big( L(p_0)\cup L(p_\infty) \big)
\end{equation*}
consists of four quadrilateral regions (see Figure~\ref{fig:einoo}).
In \S\ref{sec:crooked} the union $S$ of two non-adjacent quadrilateral
regions will be studied; this is the {\em stem\/} of a crooked
surface. Such a set is bounded by the four photons 
of $L(p_0) \cup L(p_\infty)$.

\begin{figure}[ht]
\begin{center}
\includegraphics[width=2.5cm, height=2.5cm]{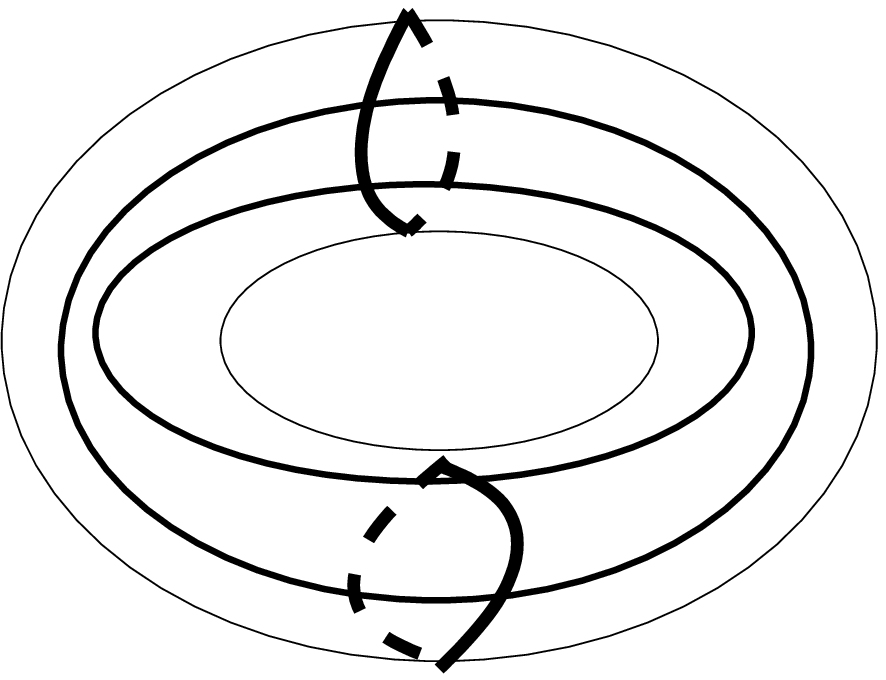}
\end{center}
\caption{Two lightcones in $\EinOO$} \label{fig:einoo}
\end{figure}


\subsection{$3$-dimensional case} \label{sec:threedim}
Here we present several observations particular to the 
case of $\EinTO$.

\begin{itemize}

\item We will see that $\Pho$ identifies naturally
with a $3$-dimensional real projective space (\S\ref{sec:contact}).

\item A lightcone in $\EinTO$ is homeomorphic to a
pinched torus.

\item  
Suppose $p \neq q$. Define
\begin{equation*}
C(p,q) := L(p) \cap L(q) .
\end{equation*}
If $p$ and $q$ are {\em incident,\/}---that is, they lie on a common photon---then $C(p,q)$ is the unique photon containing them.  Otherwise $C(p,q)$ is a submanifold that we will call a {\em spacelike circle.\/}
Spacelike circles are projectivized nullcones of linear subspaces of
$\RR^{3,2}$ of type $(2,1)$. The closure of a spacelike
geodesic in $\Eto$ is a spacelike circle.  \index{spacelike!circle}

\item  A {\em timelike circle\/} is the projectivized nullcone of 
a linear subspace of $\RR^{3,2}$ of metric type $(1,2)$.

\index{timelike!circle}

\item 
Einstein hyperspheres in $\EinTO$ are copies of $\EinOO$.  In addition to their two rulings by photons, they have a foliation by spacelike circles.

\item
Lightcones may intersect Einstein hyperspheres in two different ways.
These correspond to intersections of degenerate linear hyperplanes in 
$\RR^{3,2}$ with linear hyperplanes of type $(2,2)$.  Let $u,v\in\RR^{3,2}$ be vectors such that
$u^\perp$ is degenerate, so $u$ determines a lightcone $L$, and $v^\perp$ has type $(2,2)$, so $v$ defines the Einstein hypersphere $H$.
In terms of inner products,
\begin{equation*}
\langle u,u\rangle = 0, \  
\langle v,v\rangle > 0 .
\end{equation*}
If $\langle u,v\rangle \neq 0$, then $u,v$ span a nondegenerate subspace
of signature $(1,1)$. In that case $L\cap H$ is a spacelike circle.
If $\langle u,v\rangle = 0$, then $u,v$ span a degenerate subspace
and the intersection is a lightcone in $H$, which is a union of two
distinct but incident photons.

\item
Similarly, lightcones intersect spacelike hyperspheres in two different ways.
The generic intersection is a spacelike circle, and the non-generic
intersection is a single point, such as the intersection of $L(0)$ with 
the spacelike plane $z = 0$ in $\Rto$.

\item 
A {\em pointed photon\/} is a pair $(p,\phi)\in\EinTO\times\Pho$ such
that $p\in\phi$. Such a pair naturally extends to a triple
\begin{equation*}
p \in \phi \subset L(p)
\end{equation*}
which corresponds to an {\em isotropic flag,\/} that is, a linear
filtration of $\RR^{3,2}$
\begin{equation*}
0 \subset \ell_p \subset P_\phi \subset (\ell_p)^\perp \subset \RR^{3,2},
\end{equation*}
where $\ell_p$ is the $1$-dimensional linear subspace corresponding
to $p$; $P_\phi$ is the $2$-dimensional isotropic subspace corresponding
to $\phi$; and $(\ell_p)^\perp$ is the orthogonal subspace of $\ell_p$. These objects
form a homogeneous space, an incidence variety, denoted $\Flag$, of $\Oh(3,2)$, which fibers both over $\EinTO$ and $\Pho$.
The fiber of the fibration $\Flag\longrightarrow\EinTO$ over a point $p$
is the collection of all photons through $p$. 
The fiber of the fibration $\Flag\longrightarrow\Pho$ over a photon $\phi$
identifies with all the points of $\phi$. Both fibrations are circle bundles.
\index{pointed photon}
\index{flag!isotropic}
\end{itemize}

\section{$\Einno$ as the conformal compactification of $\Eno$}
\label{sec.confcomp}

Now we shall describe the geometry of $\Einno$ as the compactification
of Minkowski space $\Eno$. We begin with the Euclidean analog.

\subsection{The conformal Riemannian sphere}
\label{sub.conformalsphere}

The standard conformal compactification of Euclidean space $\EE^n$ is
topologically the one-point compactification, the $n$-dimensional
sphere.  The {\em conformal Riemannian sphere\/} $S^{n}$ is the
projectivization $\Pr(\Nno)$ of the nullcone of $\RR^{n+1,1}$.   

For $U\subset S^{n}$ an arbitrary open set, any local section
\begin{equation*}
U \xrightarrow{\sigma} \RR^{n+1,1} \setminus \{ 0 \}
\end{equation*}
of the restriction of the projectivization map to $U$ determines a pullback of the Lorentz metric 
on $\EE^{n+1,1}$ to a Riemannian metric $g_{\sigma}$ on $U$.
This metric depends on $\sigma$, but its conformal class is independent
of $\sigma$. 
Every section is $\sigma' = f\sigma$ for some non-vanishing function $f: U
\to \RR$.  Then
\begin{equation*}
g_{\sigma'} = f^{2}g_{\sigma}
\end{equation*}
so the pullbacks are conformally equivalent.
Hence the metrics $g_{\sigma}$ altogether define a canonical conformal
structure on $S^{n}$.

The orthogonal group $\Oh(n+1,1)$ leaves invariant the nullcone
$\Nno\subset\RR^{n+1,1}$.  The projectivization 
\begin{equation*}
S^{n} = \Pr(\Nno) 
\end{equation*} 
is invariant under the {\em projective orthogonal group\/} $\POh(n+1,1)$,
which is its conformal automorphism group.

Let 
\begin{equation*}
S^{n} \xrightarrow{\sigma} \Nno \subseteq \RR^{n+1,1} \setminus \{ 0 \} 
\end{equation*}
be the section taking
values in the unit Euclidean sphere. Then the metric $g_{\sigma}$ is
the usual $\Oh(n+1)$-invariant spherical metric. 

Euclidean space $\EE^n$ embeds in $S^n$
via a {\em spherical paraboloid\/} in the nullcone $\NN^{n+1,1}$.  Namely
consider the quadratic form on $\RR^{n+1,1}$ defined by

\begin{equation*}
\Id_n \oplus - 1/2 \cdot \bmatrix 0 & 1 \\ 1 & 0 \endbmatrix =
\begin{bmatrix} 
\Id_n &   & \\
      & 0 & - 1/2 \\
      & - 1/2 & 0 \end{bmatrix} .
\end{equation*}
The map
\begin{align}\label{eq:paraboloid}
\EE^n & \rightarrow \NN^{n+1,1}\subset \RR^{n+1,1} \nonumber \\
x     & \longmapsto \bmatrix x \\ \langle x,x\rangle \\ 1 \endbmatrix
\end{align}
composed with  projection $\NN^{n+1,1} \xrightarrow{\Pr} S^{n}$ is an
embedding $\mathcal E$ of $\EE^{n}$ into $S^{n}$, which is {\em conformal.\/}


The Euclidean similarity transformation
\index{similarity transformation!Euclidean}
\begin{equation*}
f_{r,A,b}: x \longmapsto r A x + b 
\end{equation*}
where $r\in \RR_+$, $A\in \Oh(n)$, and $b\in \RR^n$, is represented by
\begin{equation}\label{eq:similaritymatrix}
F_{r,A,b} :=
\begin{bmatrix} 
\Id_n & 0 & b \\
 2b^\dag  & 1 & \langle b,b\rangle \\
0  & 0 & 1 \end{bmatrix}
\cdot 
\begin{bmatrix} 
A & 0 & 0 \\
0 & r & 0 \\
0  & 0 & r^{-1} \end{bmatrix} \;
\in\;  \Oh(n+1,1) .
\end{equation}

That is, for every $x\in\EE^n$, 
\begin{equation*}
F_{r,A,b} \Ee(x) = \Ee\big(f_{r,A,b}(x)\big) .
\end{equation*}
{\em Inversion in the unit sphere \/}
$\langle v,v \rangle = 1$ of $\EE^n$ is represented by the element 
\begin{equation*}
\Id_n \oplus \bmatrix 0 & 1 \\ 1 & 0 \endbmatrix
\end{equation*}
which acts on $\EE^n\setminus\{0\}$ by:
\begin{equation*}
\iota: x \mapsto \frac1{\langle x,x\rangle} x .
\end{equation*}
The origin is mapped to the point (called $\infty$) 
having homogeneous coordinates
\begin{equation*}
\bmatrix 0_n \\ 1 \\ 0 \endbmatrix 
\end{equation*}
where $0_n\in \RR^n$ is the zero vector.
\index{conformal inversion!Euclidean}

The map $\Ee^{-1}$ is a coordinate chart on the open set
\begin{equation*}
\EE^n = S^n \setminus \{\infty\} 
\end{equation*} 
and $\Ee^{-1}\circ\iota$ is a coordinate chart on the open set
$(\EE^n \cup\{\infty\}) \setminus \{0\} = S^n \setminus\{0\}$. 

\subsection{The conformal Lorentzian quadric}
\label{sub.confeinstein}
Consider now the inner product space $\RR^{n+1,2}$. Here it will
be convenient to use the inner product
\begin{align*}
\langle u,v \rangle  & :=  u_1 v_1 + \ldots + u_n v_n  - u_{n+1}v_{n+1}
-  \frac12 u_{n+2}v_{n+3} -  \frac12 u_{n+3}v_{n+2}  \\
& = u^\dag  \bigg(\Id_n \oplus -\Id_1 
\oplus - 1/2 \cdot \bmatrix 0 & 1 \\ 1 & 0 \endbmatrix\bigg) v .
\end{align*}
In analogy with the Riemannian case, consider the embedding 
$\Ee: \Eno \rightarrow \Einno$ via 
a {\em hyperbolic paraboloid\/} defined by \eqref{eq:paraboloid} as above, 
where the Lorentzian inner product on $\Eno$ is defined by 
$Q = \Id_n \oplus -\Id_1$.  
The procedure used previously in the Riemannian case
naturally defines an $\Oh(n+1,2)$-invariant conformal Lorentzian
structure on $\Einno$, and the embedding we have just defined is
conformal.

{\em Minkowski similarities\/} $f_{r,A,b}$ map into $\Oh(n+1,2)$ as in
the formula \eqref{eq:similaritymatrix}, where 
$r \in \RR_+; A \in \Oh(n+1,1); b \in \RR^{n,1}$; $\langle,\rangle$
is the Lorentzian inner product on $\RR^{n,1}$; and $2b^\dag$ is replaced by $2b^\dag Q$.


The conformal compactification of Euclidean space is the one-point
compactification; the compactification of Minkowski space, however, is
more complicated, requiring the addition of more than a single
point. 
Let $p_0 \in \Einno$ denote the {\em origin,\/} corresponding to
\begin{equation*}
\bmatrix 0_{n+1} \\ 0 \\ 1 \endbmatrix .
\end{equation*}
To see what lies {\em at infinity,\/}
consider the {\em Lorentzian inversion in the unit sphere\/}
\index{conformal inversion!Lorentzian}
defined by the matrix $\Id_{n+1}\oplus \bmatrix 0 & 1 \\ 1 & 0 \endbmatrix$,
which is given on $\Eno$ by the formula
\begin{equation}\label{eq:LorentzianInversion}
\iota: x \longmapsto \frac1{\langle x,x\rangle} x .
\end{equation}
Here the whole affine lightcone $L^{\aff}(p_0)$ is thrown to 
infinity. We distinguish the points on $\iota(L^{\aff}(p_0))$:
\begin{itemize} 
\item The {\em improper point\/} $p_\infty$ is the image $\iota(p_0)$.
It is represented in homogeneous coordinates by
\begin{equation*}
\bmatrix 0_{n+1} \\ 1 \\ 0 \endbmatrix .
\end{equation*}
\item The {\em generic point\/} on 
$\iota\big(L^{\aff}(p_0)\big)$ has homogeneous coordinates
\begin{equation*}
\bmatrix  v \\ 1 \\ 0 \endbmatrix
\end{equation*}
where $0\neq v\in \RR^{n,1}$; it equals $\iota\big(\Ee(v)\big)$.
\end{itemize}
\index{improper point}

We have described all the points in  
\begin{equation*}
\Eno \cup \iota(\Eno)
\end{equation*}
which are the points defined by vectors $v\in \RR^{n+1,2}$ with
coordinates $v_{n+2} \neq 0$ or $v_{n+3} \neq 0$. It remains to 
consider points having homogeneous coordinates
\begin{equation*}
\bmatrix v \\ 0 \\ 0 \endbmatrix
\end{equation*}
where necessarily $\langle v, v \rangle =0$. 
This equation describes the nullcone in $\RR^{n,1}$; its projectivization is a spacelike
sphere $S_\infty$, which we call the {\em ideal sphere.\/} 
When $n=2$, we call this the {\em ideal circle} and its elements \emph{ideal points}.
Each ideal point is the endpoint of a unique null geodesic from the origin; the union of that null
geodesic with the ideal point is a photon through the origin. Every
photon through the origin arises in this way.  
The ideal sphere is fixed by the inversion $\iota$. 

\index{ideal sphere}

The union of the ideal sphere $S_\infty$ with $\iota(L^{\aff}(p_0))$
is the lightcone $L(p_\infty)$ of the improper point.  Photons in $L(p_\infty)$ are called \emph{ideal photons.}
Minkowski space $\Eno$ is thus the complement of a lightcone 
$L(p_\infty)$ in $\Einno$. This fact motivated the earlier definition of
a Minkowski patch $\Min(p)$ as the complement in $\Einno$
of a lightcone $L(p)$.

Changing a Lorentzian metric by a non-constant scalar factor modifies timelike and spacelike geodesics, but not images of null
geodesics (see for example \cite{beem}, p.~307). 
Hence the notion of (non-parametrized) null geodesic is
well-defined in a conformal Lorentzian manifold.
For $\Einno$, the null geodesics are photons.

\subsection{Involutions}
\label{sub.conformalinv}
When $n$ is even, involutions in 
$\SOh(n+1,2) \cong \operatorname{PO}(n+1,2)$ 
correspond to nondegenerate splittings of $\RR^{n+1,2}$.
For any involution in
$\POh(3,2)$, the fixed point set in $\EinTO$ must be one of the following:
\begin{itemize}
\item the empty set $\emptyset$;
\item a  spacelike hypersphere;
\item a timelike circle;
\item the union of a spacelike circle with two points;
\item an Einstein hypersphere.
\end{itemize}
In
the case that $\operatorname{Fix}(f)$ is disconnected and equals
\begin{equation*}
\{ p_1, p_2\} \cup S 
\end{equation*}
where $p_1,p_2\in\EinTO$, and $S\subset \EinTO$ is a spacelike circle,
then 
\begin{equation*}
S = L(p_1) \cap L(p_2) .
\end{equation*}
Conversely, given any two non-incident points $p_1,p_2$, there is a 
unique involution fixing $p_1,p_2$ and the spacelike circle
$L(p_1) \cap L(p_2)$.


\subsubsection{Inverting photons}
Let $p_\infty$ be the improper point, as above.  A photon in $\EinTO$
either lies on the ideal lightcone $L(p_\infty)$, or it intersects the
spacelike plane $S_0$ consisting of all
\begin{equation*}
p =  \bmatrix x \\ y \\ z \endbmatrix
\end{equation*}
for which $z=0$. Suppose $\phi$ is a photon intersecting $S_0$ 
in the point $p_0$
with polar coordinates
\begin{equation*}
p_0 = \bmatrix r_0 \cos(\psi) \\ r_0 \sin(\psi) \\ 0 \endbmatrix \in
S_0\subset \EE^{2,1} .
\end{equation*}
Let $v_0$ be the null vector
\begin{equation*}
v_0 = \bmatrix \cos(\theta) \\ \sin(\theta) \\ 1 \endbmatrix 
\end{equation*}
and consider the parametrized lightlike geodesic
\begin{equation*}
\phi(t) :=  p_0 + t v_0
\end{equation*}
for $t\in\RR$. Then inversion $\iota$ maps $\phi(t)$ to
\begin{equation*}
(\iota\circ\phi)(t) = \iota(p_0) + \tilde{t} 
\bmatrix -\cos(\theta - 2\psi) \\
\sin(\theta - 2\psi) \\ 1 \endbmatrix
\end{equation*}
where
\begin{equation*}
\tilde{t} := \frac{t}{r_0^2 + 2 r_0 \cos(\theta-\psi) t} ~.
\end{equation*}
Observe that $\iota$ leaves invariant the spacelike plane $S_0$
and acts by Euclidean inversion on that plane.

\subsubsection{Extending planes in $\Eto$ to $\EinTO$}
\begin{itemize}
\item
The closure of a null plane $P$ in $\EE^{2,1}$ is a lightcone
and its frontier $\bar{P} \setminus P$ is an ideal photon.  
Conversely 
a lightcone with vertex on the ideal circle $S_\infty$
is the closure of a null plane containing $p_0$, while a lightcone with vertex on 
\begin{equation*}
L(p_\infty) \;\setminus\; (S_\infty \cup \{p_\infty\}) 
\end{equation*}
is the closure of a null plane not containing $p_0$.

\item
The closure of a spacelike plane in $\EE^{2,1}$ is
a spacelike sphere and its frontier is the improper point $p_\infty$.

\item
The closure of a timelike plane in $\EE^{2,1}$ is
an Einstein hypersphere and its frontier is a union of two ideal photons
(which intersect in $p_\infty$).

\item
The closure of a timelike (respectively spacelike) 
geodesic in $\EE^{2,1}$ is a timelike (respectively spacelike) 
circle containing $p_\infty$, and $p_\infty$ is its frontier.
\end{itemize}

Consider the inversion on the lightcone of $p_0$: 
\begin{equation*}
\iota \left( \begin{bmatrix} 
t \sin\theta \\ t\cos\theta\\t \\ 0\\1
\end{bmatrix} \right) 
= \begin{bmatrix} 
t \sin\theta \\ t\cos\theta\\t \\ 1\\0
\end{bmatrix} .
\end{equation*}
The entire image of the light cone 
$L(p_0)$ lies outside the Minkowski patch $\EE^{2,1}$.

Let us now look at the image of a timelike line in $\EE^{2,1}$
under the inversion. For example,
$$
\iota \left( \begin{bmatrix} 0\\0\\t\\-t^2\\1 \end{bmatrix} \right)=
 \begin{bmatrix} 0\\0\\t\\1\\-t^2 \end{bmatrix}
\sim \begin{bmatrix} 0\\0\\-1/t\\-1/t^2\\1 \end{bmatrix} =
\begin{bmatrix} 0\\0\\s\\-s^2\\1 \end{bmatrix}
$$
where $s=-1/t$. That is, the inversion maps the timelike line
minus the origin to itself, albeit with a change in the
parametrization.

\section{Causal geometry}
\label{sec.causal}
In \S\ref{sub.confeinstein} we observed that $\Einno$ is
naturally equipped with a conformal structure.  This structure lifts to the double cover $\uEinno$.  As in the Riemannian case in \S\ref{sub.conformalsphere},
a global representative of the conformal structure on $\uEinno$ is the pullback by
a global section $\sigma: \uEinno \to \RR^{n+1,2}$ of the ambient
quadratic form of $\RR^{n+1,2}$. 
%
The section $\sigma:
\widehat{\Ein}^{n,1} \to \RR^{n+1,2}$ taking values in the set where
\begin{equation*}
v_1^2 + \dots + v_{n+1}^2 = v_{n+2}^2  + v_{n+3}^2 = 1
\end{equation*}
exhibits a homeomorphism $\uEinno \cong S^n \times S^1$ as in \S\ref{sub.defeinstein}; it is now apparent that
$\widehat{\Ein}^{n,1}$ is conformally equivalent to $S^{n} \times S^{1}$ endowed with the Lorentz metric $ds_{0}^{2} -
d\theta^{2}$, where $ds_{0}^{2}$ and $d\theta^{2}$ are the usual round
metrics on the spheres $S^{n}$ and $S^{1}$ of radius one.

In the following, elements of $S^{n} \times S^{1}$ are denoted by $(\varphi, \theta)$.
In these coordinates, we distinguish the timelike vector field
$\eta = \partial_{\theta}$ tangent to the fibers $\{ \ast \} \times S^{1}$.

\subsection{Time orientation}
\index{time orientation} First consider Minkowski space
$\EE^{n,1}$ with underlying vector space $\RR^{n,1}$ equipped with the
inner product:
\begin{equation*}
\langle u,v \rangle  
:=
u_1 v_1 + \dots + u_n v_n - u_{n+1} v_{n+1} .
\end{equation*}
A vector $u$ in $\RR^{n,1}$ is \emph{causal} if $u_{n+1}^{2} \geq u_{1}^{2} +
\ldots + u_{n}^{2}$. It is {\em future-oriented}
(respectively {\em past-oriented}) if the coordinate $u_{n+1}$ is
positive (respectively negative);  
\index{future!-oriented vector}
\index{past!-oriented vector}
equivalently, $u$ is future-oriented if its inner product with
\begin{equation*}
\eta_{0} = \bmatrix 0 \\ \vdots \\ 0 \\ 1 \endbmatrix
\end{equation*}
is negative.

The key point is that the choice of the coordinate
$u_{n+1}$---equivalently, of an everywhere timelike vector field like
$\eta_{0}$---defines a decomposition of every affine lightcone
$L^{\aff}(p)$ in three parts:

\begin{itemize}
\item $\{ p \}$;
\item The future lightcone $L_{+}^{\aff}(p)$ of elements $p + v$ where
$v$ is a future-oriented null vector;
\index{future!lightcone}
\item The past lightcone $L_{-}^{\aff}(p)$ of elements $p + v$ where
$v$ is a past-oriented null vector.
\index{past!lightcone}
\end{itemize}

The above choice is equivalent to a continuous choice of one of the
connected components of the set of timelike vectors based at each $x
\in \EE^{n,1}$; timelike vectors in these components are designated future-oriented.  In other words, $\eta_{0}$ defines a
\emph{time orientation} on $\EE^{n,1}$.

To import this notion to $\uEinno$, replace
$\eta_{0}$ by the vector field $\eta$ on $\uEinno$. Then a causal
tangent vector $v$ to $\uEinno$ is {\em
future-oriented}\index{future!-oriented tangent vector} (respectively {\em
past-oriented})\index{past!-oriented tangent vector} if the inner product
$\langle v , \eta \rangle$ is negative (respectively positive).

We already observed in \S\ref{sub.defeinstein} that the antipodal map
is $(\varphi, \theta) \mapsto (-\varphi, -\theta)$ on $S^{n} \times
S^{1}$; in particular, it preserves the timelike vector field $\eta$,
which then descends to a well-defined vector field on $\Einno$, so that
$\Einno$ is time oriented, for all integers $n$.

\begin{remark}
The Einstein universe does not have
a preferred Lorentz metric in its conformal class.  The definition
above is nonetheless valid since it involves only signs of inner
products and hence is independent of the choice of metric in
the conformal class.
\end{remark}

The group $\operatorname{O}(n+1,2)$ has four connected components.
\index{orientation}
More precisely, let $\SOh(n+1,2)$ be the subgroup of
$\operatorname{O}(n+1,2)$ formed by elements with determinant $1$; these are the orientation-preserving conformal
transformations of $\uEin^{n,1}$. 
Let $\operatorname{O}^{+}(n+1,2)$ be the subgroup comprising the 
elements preserving the
time orientation of $\uEin^{n,1}$. 
The identity component of $\operatorname{O}(n+1,2)$ is the intersection
\begin{equation*}
\operatorname{SO}^{+}(n+1,2) = \SOh(n+1,2) \cap
\operatorname{O}^{+}(n+1,2) .
\end{equation*}
Moreover, $\SOh(n+1,2)$ and $\operatorname{O}^{+}(n+1,2)$ each have two
connected components.

The center of $\operatorname{O}(n+1,2)$ 
has order two and is generated by the antipodal map, which belongs
to $\SOh(n+1,2)$ if and only if $n$ is odd. 
Hence the center of
$\operatorname{SO}(n+1,2)$ is trivial if $n$ is even---in particular,
when $n=2$. On the other hand, the antipodal map always preserves the
time orientation.

The antipodal map is the only element of
$\operatorname{O}(n+1,2)$ acting trivially on $\Ein^{n,1}$. Hence the
group of conformal transformations of $\Ein^{n,1}$ is
$\operatorname{PO}(n+1,2)$, the quotient of $\operatorname{O}(n+1,2)$
by its center.  When $n$ is even, $\operatorname{PO}(n+1,2)$ is
isomorphic to $\operatorname{SO}(n+1,2)$.

\subsection{Future and past}
\label{sub.futur}

A $C^{1}$-immersion 
\begin{equation*}
[0,1] \xrightarrow{c} \EE^{1,n} 
\end{equation*}
is {\em a causal curve\/}\index{causal!curve} 
(respectively {\em a timelike curve}\index{timelike!curve}) 
if the tangent vectors $c'(t)$ are all causal (respectively timelike). 
This notion extends to any conformally
Lorentzian space---in particular, to $\Einno$, $\uEinno$, or
$\uuEinno$.  Furthermore, a causal curve $c$ is
future-oriented (respectively past-oriented) if all the tangent
vectors $c'(t)$ are future-oriented (respectively past-oriented).
\index{future!-oriented causal curve}
\index{past!-oriented causal curve}

Let $A$ be a subset of $\EE^{n,1}$, $\Einno$, $\uEinno,$ or
$\uuEinno$.  The {\em future}\index{future!of a point} $\II^{+}(A)$ (respectively
the {\em past}\index{past!of a point} $\II^{-}(A)$) of $A$ is the set comprising
endpoints $c(1)$ of future-oriented (respectively past-oriented)
timelike curves with starting point $c(0)$ in $A$. The {\em causal
future}\index{future!causal, of a point} $\JJ^{+}(A)$ (respectively the {\em
causal past}\index{past!causal, of a point} $\JJ^{-}(A)$) of $A$ is the set
comprising endpoints $c(1)$ of future-oriented (respectively
past-oriented) causal curves 
with starting point $c(0)$ in $A$.
Two points $p$, $p'$ are {\em causally related\/}
\index{causally related} 
if one belongs to the causal future of the other: $p' \in J^{\pm}(p)$.  The notion of future and past in $\EE^{n,1}$
is quite easy to understand: $p'$ belongs to the future $\II^{+}(p)$ of $p$
if and only if $p'-p$ is a future-oriented timelike element of $\RR^{n,1}$.

Thanks to the conformal model, these notions are also quite easy to 
understand in $\Einno$, $\uEinno,$ or $\uuEinno$: let $d_{n}$ be the spherical distance on the homogeneous Riemannian sphere $S^{n}$ of radius $1$.
The universal covering $\uuEinno$ is conformally isometric to the Riemannian
product $S^{n} \times \RR$ where the real line $\RR$ is endowed with
the negative quadratic form $-d\theta^{2}$.
Hence, the image of any causal, $C^{1}$, immersed curve in
$\uuEinno \approx S^{n} \times \RR$ is the graph of a map $f: I \to
S^{n}$ where $I$ is an interval in $\RR$ and where $f$ is {\em
$1$-Lipschitz \/}---that is, for all $\theta$, $\theta'$ in $\RR$:

\begin{equation*}
d_{n}(f(\theta), f(\theta')) \leq | \theta - \theta' | .
\end{equation*}
Moreover, the causal curve is timelike if and only if the map $f$ is
{\em contracting}---that is, satisfies
\begin{equation*}
d_{n}(f(\theta), f(\theta')) < | \theta - \theta' | .
\end{equation*}
It follows that the future of an element $(\varphi_{0}, \theta_{0})$ of
$\uuEinno \approx S^{n} \times \RR$ is:
\begin{equation*}
 \II^{+}(\varphi_{0}, \theta_{0}) = \{ (\varphi, \theta) \ \mid \
 \theta - \theta_{0} > d_{n}(\varphi, \varphi_{0}) \}
\end{equation*}
and the causal future $\JJ^{+}(p)$ of an element $p$ of
$\uuEinno$ is the closure of the future $\II^{+}(p)$:

\begin{equation*}
 \JJ^{+}(\varphi_{0}, \theta_{0}) = \{ (\varphi, \theta) \ \mid \
 \theta - \theta_{0} \geq d_{n}(\varphi, \varphi_{0}) \} .
\end{equation*}

As a corollary, the future $\II^{+}(A)$ of a nonempty subset $A$ of
$\Einno$ or $\uEinno$ is the entire spacetime. In other words, the notion of past or
future is relevant in $\uuEinno$, but not in $\Einno$ or
$\uEinno$. 

There is, however, a relative
notion of past and future still relevant in $\uEinno$ that will be
useful later when considering crooked planes and surfaces: let $\hat{p}$, $\hat{p}'$ be two elements of $\uEinno$ such that $\hat{p}' \neq \pm \hat{p}$.
First observe that the intersection $\Min^{+}(\hat{p}) \cap \Min^{+}(\hat{p}')$
is never empty. Let $p_\infty$ be any element of this intersection, so $\Min^{+}(\hat{p}_{\infty})$ contains $\hat{p}$ and $\hat{p}'$. 
The time orientation on $\uEinno$ induces a time orientation on such a Minkowski patch $\Min^{+}(\hat{p}_{\infty})$.

\begin{fact}
The points $\hat{p}'$ and $\hat{p}$ are causally related in $\Min^{+}(\hat{p}_{\infty})$ 
if and only if, for any lifts $p,p'$ of $\hat{p}, \hat{p}'$, respectively, to $\RR^{n+1,2}$, 
the inner product $\langle p, p' \rangle$ is positive. 
\index{Minkowski patch!positive}
\end{fact}

Hence, if $\hat{p}$ and $\hat{p}'$ are causally
related in some Minkowski patch, then they are causally related in any
Minkowski patch containing both of them. Therefore, 
(slightly abusing language)
we use the following convention: two
elements $\hat{p}$, $\hat{p}'$ of $\uEinno$ are {\em causally
related}\index{causally related} if the inner product $\langle p, p'
\rangle$ in $\RR^{n+1,2}$ is positive for any lifts $p,p'$.

\subsection{Geometry of the universal covering}
The geometrical understanding of the embedding of Minkowski space in
the Einstein universe can be a challenge. In particular,
the closure in $\Einno$ of a subset of a Minkowski patch may be not
obvious, as we will see for crooked planes. This difficulty arises
from the nontrivial topology of $\Einno$.

On the other hand, the topology of the universal covering $\uuEinno$
is easy to visualize; indeed, the map 
\begin{align*}
\uuEinno \approx S^{n} \times \RR & \xrightarrow{\mathcal S} 
\RR^{n+1} \setminus \{ 0 \} \\
\mathcal{S} : (\varphi, \theta) &\longmapsto \exp(\theta)\varphi
\end{align*}
is an embedding.  Therefore, $\uuEinno$ can be considered as a subset of
$\RR^{n+1}$---one that is particularly easy to visualize when 
$n = 2$. Observe that the map $\mathcal S$ is 
$\Oh(n+1)$-equivariant for the natural actions on $\uuEinno$ and $\RR^{n+1}$.

The antipodal map 
\begin{equation*}
(\varphi, \theta) \longmapsto (-\varphi, -\theta) 
\end{equation*}
lifts to the automorphism $\alpha$ of 
\begin{equation*}
\uuEinno \approx S^{n} \times \RR, 
\end{equation*}
defined by 
\begin{equation*}
(\varphi, \theta) \stackrel{\alpha}\longmapsto (-\varphi, \theta + \pi). 
\end{equation*}
In the coordinates $\uuEinno \approx \RR^{n+1} \setminus \{ 0 \}$ this lifting
$\alpha$ is expressed by $x \to -\lambda{x}$,
where $\lambda = \exp(\pi)$.

Since null geodesics in $\Einno$ are photons, the images by $\mathcal S$ of null geodesics 
of $\uuEinno$ are curves in
$\RR^{n+1} \setminus \{ 0 \}$ characterized by the following properties:
\begin{itemize}
\item{They are contained in $2$-dimensional linear subspaces;}
\item{Each is a logarithmic spiral in the $2$-plane containing it.}
\end{itemize}

Hence, for $n=2$, the lightcone of an element $p$ of $\uuEinTO$ (that is, the union of the null geodesics
containing $p$) is a singular surface of revolution in $\RR^{3}$ obtained by rotating a spiral contained in a vertical $2$-plane around an axis of the plane.
In particular, for every $x$ in $\uuEinno \approx \RR^{n+1} \setminus \{ 0 \}$, every null geodesic
containing $x$ contains $\alpha(x) = -\lambda{x}$.  The image $\alpha(x) = -\lambda{x}$ is uniquely characterized by the following properties, so that it can be called the {\em first future-conjugate point} to $x$:
\index{future!-conjugate point}
\begin{itemize}
\item{It belongs to the causal future $J^{+}(x)$;}
\item{For any $y \in J^+(x)$ such that $y$ belongs to all null geodesics containing $x$, we have $\alpha(x) \in J^-(y).$}
\end{itemize}

All these considerations allow us to visualize how Minkowski patches embed
in $\RR^{n+1} \setminus \{ 0 \}$ (see Figure~\ref{fig.spiral}): 
let $\tilde{p}\in\uuEinno$ and $\hat{p}$ be its projection to $\uEinno$. The 
Minkowski patch $\Min^{+}(\hat{p})$ is the projection in $\uEinno$ of 
$I^{+}(\tilde{p}) \setminus J^{+}(\alpha(\tilde{p}))$, which can also
be defined as $I^{+}(\tilde{p}) \cap I^{-}(\alpha^{2}(\tilde{p}))$. 
The projection in $\uEinno$ of 
$$\uuEinno \setminus (J^{+}(\tilde{p}) \cup J^{-}(\tilde{p}))$$
is the Minkowski patch $\Min^{-}(\hat{p})$, which is the set of points non-causally related to $\hat{p}$.

\begin{figure}[ht]
\hspace{1in} { \input{sspiral.pstex_t} }
\caption{\label{fig.spiral}
   \textit{A Minkowski patch in $\uuEin^{1,1}$}} 
\end{figure}

\subsection{Improper points of Minkowski patches}
\label{sec:improperpatch}
We previously defined the improper point $p_{\infty}$ associated to a
Minkowski patch in $\Einno$: it is the unique point such that the
Minkowski patch is $\Min(p_{\infty})$.

In the double-covering $\uEinno$, to every Minkowski patch are attached two {\em improper points}:
\begin{itemize}
\item{the {\em spatial improper point}, the unique element $p_{\infty}^{\operatorname{sp}}$ such that the given Minkowski patch is $\Min^{-}(p_{\infty}^{\operatorname{sp}})$;
\index{improper point!spatial} }
\item{ the {\em timelike improper point}, the unique element $p_{\infty}^{\operatorname{ti}}$ such that the given Minkowski patch is $\Min^{+}(p_{\infty}^{\operatorname{ti}})$.
\index{improper point!timelike} }
\end{itemize}

Let $\Min^{+}(p_{\infty}^{\operatorname{ti}}) =
\Min^{-}(p_{\infty}^{\operatorname{sp}})$ be a Minkowski patch in
$\uEinno$. Let 
\begin{equation*}
\RR \xrightarrow{\gamma} \Min^{+}(p_{\infty}^{\operatorname{ti}}) 
\approx \EE^{n,1}  
\end{equation*}
be a geodesic. Denote by $\Gamma$ the image of $\gamma$, and by
$\bar{\Gamma}$ the closure in $\uEinno$ of $\Gamma$.
\begin{itemize}
\item If $\gamma$ is spacelike, then 
\begin{equation*}
\bar{\Gamma} = \Gamma \cup \{ p_{\infty}^{\operatorname{sp}} \} .
\end{equation*}
\item If $\gamma$ is timelike, then 
\begin{equation*}
\bar{\Gamma} = \Gamma \cup \{ p_{\infty}^{\operatorname{ti}} \} .
\end{equation*}
\item If $\gamma$ is lightlike, then 
$\bar{\Gamma}$ is a photon avoiding 
$p_{\infty}^{\operatorname{sp}}$ and $p_{\infty}^{\operatorname{ti}}$.
\end{itemize}



\section{Four-dimensional real symplectic vector spaces}
\label{sec.symplectic}
In spatial dimension $n= 2$, Einstein space $\EinTO$ admits an alternate
description as the {\em Lagrangian Grassmannian,\/} the manifold $\Lag(V)$ of
Lagrangian $2$-planes in a real symplectic vector space $V$ of dimension
$4$. There results a  
\index{Lagrangian!Grassmannian}
kind of
duality between the conformal Lorentzian geometry
of $\EinTO$ and the symplectic geometry of $\RR^4$. Photons correspond
to linear pencils of Lagrangian $2$-planes (that is, families of Lagrangian
subspaces passing through a given line). The corresponding local isomorphism
\begin{equation*}
\Sp(4,\RR) \longrightarrow  \Oh(3,2) 
\end{equation*}
manifests the isomorphism of root systems of type $B_2$ (the odd-dimensional
orthogonal Lie algebras) and $C_2$ (the symplectic Lie algebras)
of rank $2$.  We present this correspondence below.
\index{symplectic!vector space}

\subsection{The inner product on the second exterior power}

Begin with a four-dimensional vector space $V$ over $\RR$ and choose
a fixed generator 
\begin{equation*}
\vol \in \Lambda^4(V) .
\end{equation*}
The group of automorphisms of $(V,\vol)$ is the special linear group
$\SL(V)$. 

The second exterior power $\Lambda^2(V)$ has dimension $6$. The action of
$\SL(V)$ on $V$ induces an action on 
$\Lambda^2(V)$  which preserves the bilinear form
\begin{equation*}
\Lambda^2(V) \times \Lambda^2(V) \xrightarrow{\BB} \RR 
\end{equation*}
defined by:
\begin{equation*}
\alpha_1 \wedge \alpha_2 =  -\BB(\alpha_1,\alpha_2) \vol .
\end{equation*}

This bilinear form satisfies the following properties:
\begin{itemize}
\item $\BB$ is symmetric;
\item $\BB$ is nondegenerate;
\item $\BB$ is split---that is, of type $(3,3)$.
\end{itemize}
(That $\BB$ is split follows from the fact that any orientation-reversing
linear automorphism of $V$ maps $\BB$ to its negative.)


The resulting homomorphism
\begin{equation}\label{eq:secondexteriorpower}
\SL(4,\RR) \longrightarrow \SOh(3,3) 
\end{equation}
is a local isomorphism of Lie groups, with kernel $\{\pm \Id\}$ and image the identity component of $\SOh(3,3)$.


Consider a symplectic form $\omega$ on $V$---that is, a skew-symmetric
nondegenerate bilinear form on $V$.  Since $\BB$ is nondegenerate,
$\omega$ defines a dual exterior bivector $\omega^*\in\Lambda^2(V)$ by
\begin{equation*}
\omega(v_1,v_2) = \BB(v_1\wedge v_2, \omega^*) .
\end{equation*}
We will assume that
\begin{equation}\label{eq:normalizeDualBivector}
\omega^* \wedge \omega^* = 2 \vol .
\end{equation}
Thus $\BB(\omega^*,\omega^*) = -2 < 0$, so that its symplectic complement
\begin{equation*}
W_0
:= (\omega^*)^\perp\subset \Lambda^2(V) 
\end{equation*}
is an inner product space of type $(3,2)$.
Now the local isomorphism
\eqref{eq:secondexteriorpower} restricts to a local isomorphism
\begin{equation}\label{eq:secondexteriorpowersymplectic}
\Sp(4,\RR) \longrightarrow \SOh(3,2) 
\end{equation}
with kernel $\{\pm\Id_4\}$ and image the identity component of $\SOh(3,2)$.  


\subsection{Lagrangian subspaces and the Einstein universe}
\index{Lagrangian!plane}
Let $V$, $\omega$, $\BB$, $\omega^*,$ and 
$W_0$ be as above. 
The projectivization of the null cone in $W_0$ is equivalent to $\EinTO$.
Points in $\EinTO$ correspond to Lagrangian planes in $V$---that is,
$2$-dimensional linear subspaces $P\subset V$ such that the
restriction $\omega|_P \equiv 0$. Explicitly, if $v_1,v_2$ constitute
a basis for $P$, then the line generated by the bivector 
\begin{equation*}
w = v_1\wedge v_2\in\Lambda^2(V) 
\end{equation*}
is independent of the choice of basis for $P$. 
Furthermore, $w$ is null with respect to $\BB$ and orthogonal
to $\omega^*$, so $w$ generates a null line in $W_0 \cong \RR^{3,2}$,
and hence defines a point in $\EinTO$.  
 
For the reverse correspondence, first note that a point of $\EinTO \cong \Pr(\NN(W_0))$ is
represented by a vector $a \in W_0$ such that $a \wedge a = 0$.
Elements $a \in \Lambda^2 V$ with $a\wedge a = 0$ are exactly the
decomposable ones---that is, those that can be written 
$a = v_1\wedge v_2$ for $v_1, v_2 \in V$.  
Then the condition $a \perp \omega^*$ is
equivalent by construction to $\omega(v_1,v_2)=0$, so $a$ represents a
Lagrangian plane, $\mbox{span}\{ v_1,v_2 \}$, in $V$.  Thus Lagrangian $2$-planes in $V$ correspond to isotropic lines in $W_0 \cong \RR^{3,2}$.   

For a point $q \in \EinTO$, denote by $L_q$ the corresponding Lagrangian plane in $V$.

\subsubsection{Complete flags}
A photon $\phi$ in $\EinTO$ corresponds to a line $\ell_\phi$ in $V$, where 
$$ \ell_\phi = \bigcap_{p \in \phi} L_p .$$
A pointed photon $(p,\phi)$, as defined in \S\ref{sec:threedim}, corresponds to a pair of linear subspaces
\begin{equation}\label{eq:linearinclusion}
\ell_\phi \subset L_p
\end{equation}
where $\ell_\phi\subset V$ is the line corresponding to $\phi$ and 
where $L_p\subset V$ is the Lagrangian plane of corresponding to $p$. Recall that the incidence relation $p\in\phi$ extends to
\begin{equation*}
p \in \phi \subset L(p), 
\end{equation*}
corresponding to the complete linear flag
\begin{equation*}
0 \subset \ell_p \subset P_\phi \subset (\ell_p)^\perp \subset W_0  
\end{equation*}
where $P_\phi$ is the null plane projectivizing to $\phi$.
The linear inclusion \eqref{eq:linearinclusion} extends to a linear flag
\begin{equation*}
0 \subset \ell_\phi \subset L_p \subset (\ell_\phi)^\perp \subset V
\end{equation*}
where now $(\ell_\phi)^\perp$ denotes the symplectic orthogonal of $\ell_\phi$.
Clearly the lightcone $L(p)$ corresponds to the linear
hyperplane $(\ell_\phi)^\perp\subset V$.

\subsubsection{Pairs of Lagrangian planes}
\label{sub.sub.lagrangianpair}
Distinct Lagrangian subspaces $L_1,L_2$ may intersect in either a line
or in $0$. 
If $L_1\cap L_2\neq 0$,
the corresponding points $p_1,p_2\in\EinTO$ are incident.
Otherwise
\begin{equation*}
V = L_1 \oplus L_2 
\end{equation*}
and the linear involution of $V$
\begin{equation*}
\theta = \Id_{L_1} \oplus -\Id_{L_2} 
\end{equation*}
is {\em anti-symplectic \/}: 
\begin{equation*}
\omega(\theta(v_1),\theta(v_2)) = - \omega(v_1,v_2).   
\end{equation*}
The corresponding involution of $\EinTO$ fixes the two points $p_1,p_2$ and the
spacelike circle $L(p_1)\cap L(p_2)$. 
It induces a time-reversing involution of $\EinTO$.
\index{anti-symplectic!involution}

\subsection{Symplectic planes}
\index{symplectic!plane}
\index{symplectic!involution}
Let $P\subset V$ be a {\em symplectic plane,\/} that is, one for which
the restriction $\omega|_P$ is nonzero (and hence nondegenerate).
Its symplectic complement $P^\perp$ is also a symplectic plane,
and 
\begin{equation*}
V = P \oplus P^\perp 
\end{equation*}
is a symplectic direct sum decomposition. 

Choose a basis $\{u_1,u_2\}$ for $P$. 
We may assume that $\omega(u_1,u_2) = 1$.
Then 
\begin{equation*}
\BB(u_1\wedge u_2,\omega^*) = 1 
\end{equation*}
and
\begin{equation*}
\upsilon_P := 2 u_1\wedge u_2 +  \omega^* 
\end{equation*}
lies in $(\omega^*)^\perp$
since $\BB(\omega^*,\omega^*)=-2$. Furthermore
\begin{align*}
\BB(\upsilon_P,\upsilon_P) & =  \;
\BB(2 u_1\wedge u_2, 2 u_1\wedge u_2) \;+\; 2\ \BB( 2 u_1\wedge u_2, \omega^*) 
\;+\; \BB(\omega^*,\omega^*)\\ &  =\; 0\; +\; 4\; -\; 2 \\ & =\; 2.
\end{align*}
whence $\upsilon_P$ is a positive vector in $W_0 \cong \RR^{3,2}$. In particular
$\Pr(\upsilon_P^\perp\cap\NN(W_0))$ is an Einstein hypersphere.

The two symplectic involutions leaving $P$ (and necessarily also $P^\perp$)
invariant 
\begin{equation*}
\pm \big(\Id|_{P} \oplus -\Id|_{P^\perp}\big)
\end{equation*}
induce maps fixing $\upsilon_P$, and acting 
by $-1$ 
on $(\upsilon_P)^\perp$.
The corresponding eigenspace decomposition
is $\RR^{1,0}\oplus\RR^{2,2}$ and the corresponding conformal involution
in $\EinTO$ fixes an Einstein hypersphere.
\index{Einstein!hypersphere}

\subsection{Positive complex structures and the Siegel space}
Not every involution of $\EinTO$ arises from a linear involution
of $V$. Particularly important are those which arise
from {\em compatible complex structures\/}, defined as follows.
A {\em complex structure\/} on $V$ is an automorphism
$V\xrightarrow{\JJ}V$ such that 
$\JJ\circ\JJ = -\Id$. The pair $(V,\JJ)$ then inherits the structure
of a {\em complex vector space\/} for which $V$ is the underlying
real vector space. The complex structure $\JJ$ is {\em compatible\/}
with the symplectic vector space $(V,\omega)$ when
\begin{equation*}
\omega(\JJ x,\JJ y)  = \omega(x,y) .
\end{equation*}
(In the language of complex differential geometry,
the exterior $2$-form $\omega$ {\em has Hodge type \/}
$(1,1)$ on the complex vector space $(V,\JJ)$.) Moreover 
\begin{align*}
V \times V & \longrightarrow \CC \\
(v,w) & \longmapsto \omega(v,\JJ w) +  i \omega(v,w) 
\end{align*}
defines a {\em Hermitian form\/} on $(V,\JJ)$.

A compatible complex structure $\JJ$ on $(V,\omega)$ is {\em positive\/} if $\omega(v,\JJ v) > 0$ whenever $v\neq 0$.
Equivalently, the symmetric bilinear form defined by
\begin{equation*}
v \cdot w := \omega(v,\JJ w) 
\end{equation*}
is positive definite. This is in turn equivalent to the above Hermitian
form being positive definite.

The positive compatible complex structures on $V$ are parametrized by the {\em
symmetric space\/} of $\Sp(4,\RR)$. A convenient model is the {\em
Siegel upper-half space\/} $\mathfrak{S}_2$, which can be realized as
the domain of $2\times 2$ complex symmetric matrices with positive definite imaginary part (Siegel~\cite{Siegel}).
\index{Siegel upper-half space} 

A matrix $M \in \Sp(4,\RR)$ acts on a complex structure $\JJ$ by 
$$ \JJ \mapsto M \JJ M^{-1}$$
and the stabilizer of any $\JJ$ is conjugate to $U(2)$, the group of unitary transformations of $\CC^2$.
Let the symplectic structure
$\omega$ be defined by the $2\times 2$-block matrix 
\begin{equation*}
\JJ := \bmatrix 0_2 & -\Id_2 \\ \Id_2 & 0_2 \endbmatrix .
\end{equation*}
This matrix also defines a complex structure.
Write $M$ as a block matrix with 
\begin{equation*}
M = \bmatrix A & B \\ C & D \endbmatrix  
\end{equation*}
where the blocks $A,B,C,D$ are $2\times 2$ real matrices.
Because $M\in\Sp(4,\RR)$,
\begin{equation}\label{eq:symplecticmatrix}
M^\dag \JJ M = \JJ.   
\end{equation}
The condition that $M$ preserves the complex structure $\JJ$ means 
that $M$ commutes with $\JJ$, which
together with \eqref{eq:symplecticmatrix}, means that
\begin{equation*}
M^\dag M = \Id_4,
\end{equation*}
that is, $M \in \Oh(4)$.
Thus the stabilizer of the pair $(\omega,\JJ)$ is $\Sp(4,\RR)\cap\Oh(4)$,
which identifies with the unitary group $\UU(2)$ as follows.

If $M$ commutes with $\JJ$, then its block entries satisfy
\begin{equation*}
B = -C, \qquad D = A .
\end{equation*}
Relabelling $X = A$ and $Y = C$, then
\begin{equation*}
M = \bmatrix X & -Y \\ Y & X \endbmatrix 
\end{equation*}
corresponds to a complex matrix $Z = X + iY$.  
This matrix is symplectic if and only if $Z$ is unitary,
\begin{equation*}
\bar{Z}^\dag Z= \Id_2 .
\end{equation*}

\subsection{The contact projective structure on photons}\label{sec:contact}

The points of a photon correspond to Lagrangian planes in $V$
intersecting in a common line.
Therefore, photons correspond to linear 1-dimensional subspaces in $V$, 
and the photon space $\Pho$ identifies with the projective
space $\Pr(V)$. This space has a natural contact geometry defined below.

Recall that a {\em contact structure\/} on a manifold $M^{2n+1}$ is 
a vector subbundle $E\subset TM$ of codimension one that is {\em maximally non-integrable}:
$E$ is locally the kernel of a nonsingular $1$-form $\alpha$
such that $\alpha \wedge \big(d\alpha\big)^n$ is nondegenerate at every point. This condition is independent of the $1$-form $\alpha$
defining $E$, and is equivalent to the condition that any two points
in the same path-component can be joined by a smooth curve with
velocity field in $E$. The $1$-form $\alpha$ is called a {\em contact $1$-form\/} defining $E$. 
For more details on contact geometry, see \cite{McDuffSalamon,Goldman,Thurston}.

The restriction of $d\alpha$ to $E$ is a nondegenerate exterior 2-form,
making $E$ into a {\em symplectic vector bundle.\/} Such a vector bundle
always admits a compatible complex structure $J_E:E\longrightarrow E$ (an automorphism
such that $J_E\circ J_E = -\Id$), which gives $E$ the structure of a 
{\em Hermitian vector bundle.\/} The contact structure we define on
photon space $\Pr(\RR^4)\cong\Pho$ will have such Hermitian structures and
contact $1$-forms arising from compatible complex structures on the symplectic
vector space $\RR^4$.

\subsubsection{Construction of the contact structure}

Let $v\in V$ be nonzero, and denote the corresponding line by 
$[v]\in \Pr(V)$. The tangent space $T_{[v]}\Pr(V)$ naturally identifies with
$\Hom([v],V/[v])$ ($[v]\subset V$ denotes the $1$-dimensional subspace
of $V$, as well).  If $V_1\subset V$ is a hyperplane complementary to
$[v]$, then an {\em affine patch\/} for $\Pr(V)$ containing $[v]$ is
given by
\begin{align*}
\Hom([v],V_1) & \xrightarrow{A_{V_1}}  \Pr(V) \\
\phi & \longmapsto  [v + \phi(v)] .
\end{align*}
That is, $A_{V_1}(\phi)$ is the graph of the linear map $\phi$
in $V = [v] \oplus V_1$. This affine patch defines an isomorphism
\begin{equation*}
T_{[v]}\Pr(V)\longrightarrow \Hom([v],V_1) \cong \Hom([v],V/[v])  
\end{equation*}
that is independent of the choice of $V_1$.
\index{contact!structure}
Now, since $\omega$ is skew-symmetric, symplectic product with $v$ defines
a linear functional 
\begin{align*}
V/[v] & \xrightarrow{\alpha_v} \RR \\ 
u     & \longmapsto \omega(u,v).
\end{align*}
The hyperplane field 
\begin{equation*}
[v] \longmapsto \{ \varphi \ : \ \alpha_v \circ \varphi = 0 \}  
\end{equation*}
is a well-defined
\emph{contact plane field} on $\Pr(V)$.
\index{contact!plane field} 
It posseses a unique transverse orientation; we denote
a contact 1-form for this hyperplane field by $\alpha$.
\index{contact!1-form}

\subsubsection{The contact structure and polarity}
\label{sec:polarity}
The contact structure and the projective geometry of $\Pr(V)$
interact with each other in an interesting way.
If $p\in\Pr(V)$, then the contact structure at $p$ is 
a hyperplane $E_p\subset T_p\Pr(V)$. 
There is a unique projective hyperplane
$H = H(p)$ tangent to $E_p$ at $p$.
Conversely, suppose $H\subset\Pr(V)$ is a projective hyperplane.
The contact plane field is transverse to $H$ everywhere but one point,
and that point $p$ is the unique point for which $H = H(p)$.
This correspondence results from the correspondence between
a line $\ell\subset V$ 
and its symplectic orthogonal $\ell^\perp\subset V$. 

The above correspondence is an instance of a {\em polarity\/} in projective
geometry. A {\em polarity\/} of a projective space $\Pr(V)$  is a projective
isomorphism between $\Pr(V)$ and its dual $\Pr(V)^* :=\Pr(V^*)$, arising from a nondegenerate bilinear form
on $V$, which can be either symmetric or skew-symmetric.

Another correspondence is between the set of photons through a given point $p\in\EinTO$ and the set of $1$-dimensional linear suspaces of the Lagrangian plane
$L_p\subset V$. The latter set projects to a projective line in $\Pr(V)$ tangent to the contact plane field, a {\em contact projective
line.\/} All contact projective lines arise from points in $\EinTO$ in this way.

\subsubsection{Relation with positive complex structures on $\RR^4$}

A compatible positive complex structure $\JJ$ defines a {\em contact vector field\/}
for the contact structure as follows.
Let $\Omega\subset\Pr(V)$ be a subdomain.
For any nonzero $v\in V$, the map $v\longmapsto\JJ(v)$ defines an element of
$\Hom([v],V/[v])$, that is, a tangent vector in $T_{[v]}\Pr(V)$. The
resulting vector field $\xi_{\JJ}$ satisfies $\alpha(\xi_J)>0$
for any $1$-form $\alpha$ defining the contact structure,
since $\omega(v,\JJ v) > 0$ for nonzero $v\in V$.  More generally, for any
smooth map $\JJ:\Omega\longrightarrow\mathfrak{S}_2$, this
construction defines a contact vector field.
\index{contact!vector field}
\index{complex structure}

\subsection{The Maslov cycle}
\index{Maslov!cycle}
Given a $2n$-dimensional symplectic vector space $V$ over $\RR$, 
the set $\Lag(V)$ of Lagrangian subspaces of $V$ 
is a compact homogeneous space.
It identifies with 
$\operatorname{U}(n)/\Oh(n)$, given a choice of a positive compatible
complex structure on $V\cong \RR^{2n}$. 
\index{Lagrangian!Grassmannian}
The fundamental group
\begin{equation*}
\pi_1\big(\Lag(V)\big)   \cong \ZZ.
\end{equation*}
An explicit isomorphism is given by the {\em Maslov index,\/}
which associates to a loop $\gamma$ in $\Lag(V)$ an integer.
(See McDuff-Salamon~\cite{McDuffSalamon}, \S 2.4 for a general discussion.)

Let $W\in\Lag(V)$ be a Lagrangian subspace.
The {\em Maslov cycle\/} $\Mas_W(V)$ 
associated to $W$ is the subset of $\Lag(V)$ consisting of $W'$ such that
\begin{equation*}
W \cap W' \neq 0.
\end{equation*}
Although it is not a submanifold, $\Mas_W(V)$ carries a natural co-orientation
(orientation of its conormal bundle) and defines a cycle whose homology
class generates $H_{N-1}(\Lag(V),\ZZ)$ where
\begin{equation*}
N = \frac{n(n+1)}{2} = \dim\big(\Lag(V)\big) .
\end{equation*}
The Maslov index of a loop $\gamma$ is the oriented intersection number of $\gamma$
with the Maslov cycle (after $\gamma$ is homotoped to be transverse to
$\Mas_W(V)$).
\index{Maslov!index}
If $p\in\EinTO$ corresponds to a Lagrangian subspace $W\subset V$, then the Maslov cycle $\Mas_W(V)$ corresponds to the lightcone $L(p)$. (We thank A. Wienhard for this observation.)

\subsection{Summary}

We now have a dictionary between the symplectic geometry of
$\RR^4_\omega$ and the orthogonal geometry of $\RR^{3,2}$:
\index{dictionary}
\begin{table}[ht]
\begin{center}
\begin{tabular}{ | l | l  | }
\hline
Symplectic $\RR^4_\omega$ and contact $\Pr(V)$ 
& Pseudo-Riemannian $\RR^{3,2}$ and $\EinTO$ 
\\ \hline \hline
Lagrangian planes $L \subset \RR^4_\omega$ & Points $p\in\EinTO$  \\ \hline
Contact projective lines in $\Pr(V)$& Points $p\in
\EinTO$  \\ \hline
Lines $\ell \subset \RR^4_\omega$ & Photons $\phi$  \\ \hline
Hyperplanes $\ell^\perp \subset \RR^4_\omega$ & Lightcones  \\ \hline
Symplectic planes (splittings) in $\RR^4_\omega$ & Einstein hyperspheres \\ \hline
Linear symplectic automorphisms & time-preserving conformal 
automorphisms \\ \hline
Linear anti-symplectic automorphisms & time-reversing conformal 
automorphisms \\ \hline
Flags $\ell\subset L\subset\ell^\perp$ in $\RR^4_\omega$ & 
Incident pairs $p\in \phi \subset L(p)$ \\ \hline
Positive compatible complex structures  & 
Free involutions of $\EinTO$\\ \hline
Lagrangian splittings $V = L_1 \oplus L_2$  & 
Nonincident pairs of points  \\ \hline
Lagrangian splittings $V = L_1 \oplus L_2$  & 
Spacelike circles \\ \hline
\end{tabular}

\end{center}
\end{table}

\section{Lie theory of $\Pho$ and $\EinTO$}
\label{sec.liealgebra}

This section treats the structure of the Lie algebra $\spf$ and the
isomorphism with $\ott$.  We relate differential-geometric properties
of the homogeneous spaces $\EinTO$ and $\Pho$ with
the Lie algebra representations corresponding to the isotropy. This
section develops the structure theory (Cartan subalgebras, roots,
parabolic subalgebras) and relates these algebraic notions to the
synthetic geometry of the three parabolic homogenous spaces $\EinTO$,
$\Pho$ and $\Flag$.  Finally, we discuss the geometric significance of
the Weyl group of $\Sp(4,\RR)$ and $\SOh(2,3)$.

\subsection{Structure theory}
Let $V \cong \RR^4$, equipped with the symplectic form $\omega$, as above.
We consider a {\em symplectic basis\/} $e_1,e_2,e_3,e_4$ in which $\omega$ is
\begin{equation*}
\JJ = \bmatrix 
0 & -1 & 0 & 0 \\ 
1 &  0 & 0 & 0 \\ 
0 & 0 & 0 & -1 \\ 
0 & 0 & 1 & 0 \endbmatrix
\end{equation*}

\index{symplectic basis}

\newcommand{\aot}{a_{12}}
\newcommand{\ato}{a_{21}}
\renewcommand{\bot}{b_{12}}
\newcommand{\bto}{b_{21}}
\newcommand{\roo}{r_{11}}
\newcommand{\rot}{r_{12}}
\newcommand{\rto}{r_{21}}
\newcommand{\rtt}{r_{22}}

The Lie algebra $\ggg = \spf$ consists of all $4\times 4$ real matrices $M$
satisfying 
\begin{equation*}
M^\dag \JJ + \JJ M = 0, 
\end{equation*}
that is,
\begin{equation}\label{coords.sp4}
M = 
\bmatrix  
a & \aot & \roo & \rot \\
\ato & -a & \rto & \rtt \\
-\rtt & \rot & b & \bot \\
\rto & -\roo & \bto & -b 
\endbmatrix  
\end{equation}
where $a,b,a_{ij},b_{ij},r_{ij}\in\RR$. 

\subsubsection{Cartan subalgebras}
A {\em Cartan subalgebra\/} $\aa$ of $\spf$
is the subalgebra stabilizing the four coordinate lines $\RR e_i$ for
$i=1,2,3,4$, and comprises the diagonal matrices
\begin{equation*}
H(a,b) :=
\bmatrix  
a & 0 & 0 & 0 \\
0 & -a & 0 & 0 \\
0 & 0 & b & 0 \\
0 & 0 & 0 & -b 
\endbmatrix  
\end{equation*}
for $a,b\in\RR$. The calculation 
\begin{equation*}
[H, M] = 
\bmatrix  
0  & (2a)\aot & (a-b)\roo & (a+b)\rot \\
(-2a)\ato & 0  & (-a-b)\rto & (-a+b)\rtt \\
(a-b)\rtt & (-a-b)\rot & 0 & (2b)\bot \\
(a+b)\rto & (-a+b)\roo & (-2b)\bto & 0  
\endbmatrix  
\end{equation*}
implies that the eight linear functionals assigning to $H(a,b)$ the values 
\begin{equation*}
2a, -2a, 2b, -2b, a-b, a+b, -a-b, -a+b
\end{equation*}
define the {\em root system\/}
\begin{align*}
\Delta & := \{ 
(2,0),(-2,0), (0,2), (0,-2), \\
& \qquad (1,-1), (1,1), (-1, -1), (-1, 1) 
\} \;\subset\; \aa^*
\end{align*}
pictured below.
\index{root diagram! $\spf$}

\begin{figure}[h]
\begin{center}
\includegraphics[width=3cm, height=3cm]{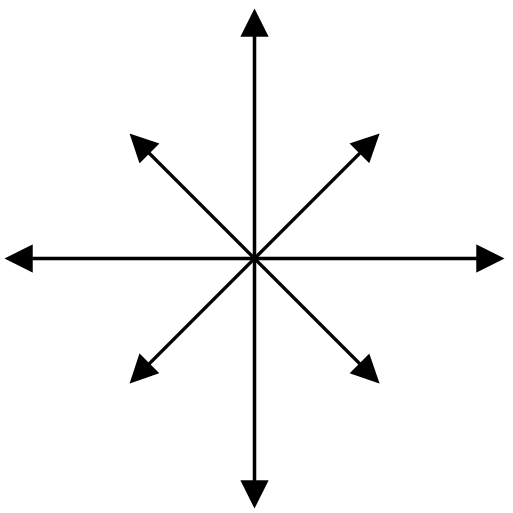}
\end{center}
\caption{Root diagram of $\spf$} \label{diagram.spf}
\end{figure}

\subsubsection{Positive and negative roots}
A vector $v_0\in\aa$ such that $\lambda(v_0)\neq 0$ for all roots
$\lambda\in\Delta$ partitions $\Delta$ into {\em positive roots\/} $\Delta_+$
and {\em negative roots\/} $\Delta_-$ depending on whether
$\lambda(v_0)>0$ or $\lambda(v_0)<0$ respectively. For example,
\begin{equation*}
v_0 = \bmatrix 1 \\ 2 \endbmatrix 
\end{equation*}
partitions $\Delta$ into
\begin{align*}
\Delta_+ & =\; \{\, (2,0), (1,1), (0,2), (-1,1)    \,\} \\ 
\Delta_- & =\; \{\, (-2,0), (-1,-1), (0,-2),(1,-1) \,\}.
\end{align*}
The positive roots
\begin{equation*}
\alpha := (2,0),\; \beta  := (-1,1)  
\end{equation*}
form a pair of {\em simple positive roots\/} in the sense that every 
$\lambda\in\Delta_+$ is a positive integral linear combination of
$\alpha$ and $\beta$. Explicitly:
\begin{equation*}
\Delta_+ = \{\alpha, \alpha + \beta, \alpha + 2\beta, \beta\}.
\end{equation*}

\subsubsection{Root space decomposition}
For any root $\lambda\in\Delta$, define the {\em root space\/}
\begin{equation*}
\ggg_\lambda  := \{ X\in \ggg \mid [H,X] = \lambda(H) X \}.
\end{equation*}
In $\ggg=\spf$,
each root space is one-dimensional, and the elements 
$X_\lambda\in\ggg_\lambda$ are 
called {\em root elements.\/} The Lie algebra decomposes as a direct sum
of vector spaces:
\begin{equation*}
\ggg =  \aa \oplus \bigoplus_{\lambda\in\Delta} \ggg_\lambda.
\end{equation*}
For more details, see Samelson~\cite{Samelson}.

\subsection{Symplectic splittings}
The basis vectors $e_1,e_2$ span a symplectic plane $P\subset V$
and $e_3,e_4$ span its symplectic complement $P^\perp\subset V$.
These planes define a {\em symplectic direct sum decomposition\/}
\begin{equation*}
V = P \oplus P^{\perp}. 
\end{equation*}
The subalgebra $\hh_P \subset \sp(4,\RR)$ preserving $P$ 
also preserves $P^\perp$
and consists of matrices of the form \eqref{coords.sp4}
that are block-diagonal:
\begin{equation*}
\bmatrix  
a & \aot & 0 & 0 \\
\ato & -a & 0 & 0 \\
0 & 0 & b & \bot \\
0 & 0 & \bto & -b 
\endbmatrix.
\end{equation*}
Thus
\begin{align*}
\hh_P &\cong \sp(2,\RR) \oplus \sp(2,\RR) \\
& \cong \sl(2,\RR) \oplus \sl(2,\RR) .
\end{align*}
The Cartan subalgebra $\aa$ of $\spf$ is also a Cartan subalgebra of
$\hh_P$, but only the four long roots 
\begin{align*}
\Delta' = \{ (\pm 2,0),(0,\pm 2)\} = 
\{\pm \alpha, \pm (\alpha + 2\beta)\}
\end{align*}
are roots of $\hh_P$. In particular $\hh_P$ decomposes as
\begin{equation*}
\hh_P =  \aa \oplus \bigoplus_{\lambda\in\Delta'} \ggg_\lambda.
\end{equation*}


\index{symplectic!plane} 

\subsection{The Orthogonal Representation of $\spf$}
\label{sub.lie.homom}
Let $e_1, \ldots, e_4$ be a symplectic basis for $V$ as above and
\begin{equation*}
\vol := e_1 \wedge e_2 \wedge e_3 \wedge e_4 
\end{equation*}
a volume element for $V$.
A convenient basis for $\Lambda^2 V$ is:
\begin{align}\label{eq:secondexteriorpowerbasis}
f_1 & := e_1 \wedge e_3 \nonumber\\ 
f_2 & := e_2 \wedge e_3 \nonumber \\ 
f_3 & := \frac{1}{\sqrt{2}}(e_1 \wedge e_2 - e_3 \wedge e_4 )\nonumber\\ 
f_4 & := e_4 \wedge e_1 \nonumber\\ 
f_5 & := e_2 \wedge e_4 
\end{align}
for which the matrix
\begin{equation*}\label{eq:secondexteriorpowerinnerproduct}
\bmatrix 
0 & 0 & 0 & 0 & 1 \\ 
0 & 0 & 0 & 1 & 0 \\ 
0 & 0 & 1 & 0 & 0 \\ 
0 & 1 & 0 & 0 & 0 \\ 
1 & 0 & 0 & 0 & 0 
\endbmatrix
\end{equation*}
defines the bilinear form $\BB$ associated to this volume element.

The matrix $M$ defined in \eqref{coords.sp4} above maps to 
\begin{equation}\label{homom.coords}
\widetilde{M} = \bmatrix   
a + b & \aot &  \rot &  -\bot & 0 \\
\ato & -a + b &   \rtt &  0 & \bot \\
\rto & \roo & 0 & -\rtt & -\rot \\
-\bto & 0 &  -\roo & a - b &  -\aot \\
0 & \bto &  -\rto & -\ato &  -a - b \endbmatrix
\;\in\; \so(3,2) .
\end{equation}

For a fixed symplectic plane $P \subset V$, such as the one spanned by
$e_1$ and $e_2$, denote by $P \wedge P^{\perp}$ the subspace of
$\Lambda^2 V$ of elements that can be written in the form $\sum_i v_i
\wedge w_i$, where $v_i \in P$ and $w_i \in P^{\perp}$ for all $i$.
The restriction of the bilinear form $\BB$ to this
subspace, which has basis $\{f_1,f_2, f_4, f_5\}$, is type $(2,2)$.  
Its stabilizer is the image $\widetilde{\hh}_P$ of $\hh_P$
in $\ott$.  Note that this image is isomorphic to 
\begin{equation*}
\oh(2,2) \cong \sl(2,\RR) \oplus \sl(2,\RR). 
\end{equation*}






\subsection{Parabolic subalgebras}
The homogeneous spaces $\EinTO$, $\Pho$ and $\Flag$ identify with 
quotients $G/P$ of $G = \Sp(4,\RR)$ where $P\subset G$ is a proper 
{\em parabolic subgroup.\/}  
When $G$ is algebraic, then any parabolic subgroup
$P$ of $G$ is algebraic, and the quotient $G/P$ is a compact
projective variety.  See Chapter~7 of \cite{Knapp} for more details.
\index{parablolic!subgroup}

As usual, working with Lie algebras is more convenient. We denote the corresponding {\em parabolic 
subalgebras\/} by $\ppp$, and they are indexed by subsets
$S\subset\Pi^-$ of the set $\Pi^- := \{-\alpha,-\beta\}$ of simple
negative roots, as follows.
\newcommand{\Shat}{\widetilde{S}}

The {\em Borel subalgebra\/} or {\em minimal parabolic subalgebra\/}
corresponds to $S = \emptyset$ and is defined as
\begin{equation*}
\ppp_\emptyset := 
\ppp_S = \aa  \oplus \bigoplus_{\lambda \in \Delta^+} \ggg_\lambda.
\end{equation*}
In general, let $\Shat$ be the set of finite sums of elements of $S$.  
The parabolic subalgebra determined by $S$ is
\begin{equation*}
\ppp_S := \ppp_\emptyset  \oplus \bigoplus_{\lambda \in \Shat} \ggg_\lambda.
\end{equation*}
\index{parabolic!subgroup} 
\index{parabolic!subalgebra} 

\subsubsection{The Borel subalgebra and $\Flag$}
Let $\ppp_\emptyset$ be the Borel subalgebra defined above.
The corresponding Lie subgroup $P_\emptyset$ is the stabilizer of a unique
pointed photon, equivalently, an isotropic flag, in $\Flag$; thus
$\Flag$ identifies with the homogeneous space $G/P_\emptyset$.
The subalgebra
\begin{equation*}
\uu_\emptyset := \sum_{\lambda\in\Delta_+}\ggg_\lambda \subset \spf
\end{equation*}
is the Lie algebra of the {\em unipotent radical\/} of $P_\emptyset$
and is 3-step nilpotent. A realization of the corresponding group is the group
generated by the translations of $\EE^{2,1}$ and a unipotent
one-parameter subgroup of $\SOh(2,1)$.

\subsubsection{The parabolic subgroup corresponding to $\Pho$}
Now let $S=\{-\alpha\}$; the corresponding parabolic subalgebra
$\ppp_{\alpha}$ is the stabilizer subalgebra of a
line in $V$, or, equivalently, of a point in $\Pr(V)$.
In $\ott$ this parabolic is the stabilizer of a null plane in $\RR^{3,2}$, or, equivalently, of
a photon in $\EinTO$. 

\index{photon} 

\subsubsection{The parabolic subgroup corresponding to $\EinTO$}
Now let $S=\{-\beta\}$; the corresponding parabolic subalgebra
$\ppp_{\beta}$ is the stabilizer subalgebra of a
Lagrangian plane in $V$, or, equivalently, a contact projective line
in $\Pr(V)$. In $\ott$, this parabolic is the stabilizer of a null line in $\RR^{3,2}$, or, equivalently, of
a point in $\EinTO$.

\subsection{Weyl groups}
\label{sub.weylgroup}

The Weyl group $W$ of $\Sp(4,\RR)$ is isomorphic to a dihedral
group of order $8$ (see Figure \ref{diagram.spf}). 
It acts by permutations on elements of the quadruples in $\Pr(V)$
corresponding to a basis of $V$.
\index{Weyl!group} 

Let $A$ be the connected subgroup of $\Sp(4,\RR)$, with
Lie algebra $\mathfrak{a}$.  In the symplectic basis $e_1, \ldots, e_4$, it
consists of matrices of the form
\begin{equation*}
\left[
\begin{array}{cccc}
a_1 & & &  \\
& a^{-1}_1 & & \\
& &  a_2 &   \\
& &  &   a_2^{-1}
\end{array}
\right]
\qquad a_1, a_2 > 0 .
\end{equation*}
The semigroup $A^+ \subset A$ with $a_2 > a_1 > 1$ corresponds to an
open Weyl chamber in $\aa$. For $i=1,2,3,4$, let $H_i$ be the
image in $\Pr(V)$ of the hyperplane spanned by $e_j$ for $j\neq i$.
The point $[e_3] \in \Pr(V)$ is an attracting fixed point for all
sequences in $A^+$, and $[e_4]$ is a repelling fixed point:
\index{Weyl!chamber} Any unbounded $a_n \in A^+$ converges
uniformly on compact subsets of $\Pr(V) \backslash H_3$ to the
constant map $[e_3]$, while $a_n^{-1}$ converges to $[e_4]$ uniformly
on compact subsets of $\Pr(V) \backslash H_2$.  On $H_3 \backslash
(H_3 \cap H_1)$, an unbounded sequence $\{ a_n \}$ converges to $[e_1]$,
while on $H_4 \backslash (H_4 \cap H_2)$, the inverses $a_n^{-1}$
converge to $[e_2]$.  


We will call the point $[e_1]$ a 
\emph{codimension-one attracting fixed point} 
for sequences in $A^+$ and $[e_2]$ a
\emph{codimension-one repelling fixed point}. 
Every Weyl chamber has associated to it a
\emph{dynamical quadruple} like $([e_3],[e_4],[e_1],[e_2] )$, 
consisting of an attracting fixed point, a repelling fixed point, 
a codimension-one attracting fixed point, and a codimension-one repelling fixed 
point.

\index{dynamical quadruple} 
\index{attracting fixed point!codimension-one} 
\index{repelling fixed point!codimension-one} 
\index{attracting fixed point} 
\index{repelling fixed point}

Conversely, given a symplectic basis $v_1, \ldots, v_4$, the intersection of the stabilizers in $\Sp(4,\RR)$ of the lines $\RR v_i$ is a Cartan subgroup $A$.  The elements $a \in A$ such that $([v_1],
\ldots, [v_4])$ is a dynamical quadruple for the sequence $a^n$ form a
semigroup $A^+$ that is an open Weyl chamber in $A$.

The Weyl group acts as a group of permutations of
such a quadruple.  These permutations must preserve a stem
configuration as in Figure~\ref{fig.configuration}, where now two
points are connected by an edge if the corresponding lines in $V$ are
in a common Lagrangian plane, or, equivalently, the two points of
$\Pr(V)$ span a line tangent to the contact structure.
The permissible permutations are those preserving the partition $\{ v_1, v_2 \} | \{ v_3, v_4 \}$.


In $\Oh(3,2)$, the Weyl group consists of permutations of four points
$p_1, \ldots, p_4$ of $\EinTO$ in a stem configuration that preserve
the configuration.
A Weyl chamber again corresponds to a dynamical quadruple $(p_1,
\ldots, p_4)$ of fixed points, where now sequences $a_n \in A^+$
converge to the constant map $p_1$ on the complement of $L(p_2)$ and
to $p_3$ on $L(p_2) \backslash (L(p_4) \cap L(p_2))$; the inverse
sequence converges to $p_2$ on the complement of $L(p_1)$ and to $p_4$
on $L(p_1) \backslash (L(p_1) \cap L(p_3))$.





\section{Three kinds of dynamics}
\label{sec.dynamic}

In this section, we present the ways sequences in $\Sp(4,\RR)$ can
diverge to infinity in terms of projective singular limits.  In
\cite{Frances2}, Frances defines a trichotomy for sequences diverging
to infinity in $\Oh(3,2)$: they have \emph{bounded}, \emph{mixed}, or
\emph{balanced distortion}.  He introduces limit sets for such
sequences and finds maximal domains of proper discontinuity for
certain subgroups of $\Oh(3,2)$.  We translate Frances' trichotomy to
$\Sp(4,\RR)$, along with the associated limit sets and maximal domains
of properness.

\subsection{Projective singular limits}
Let $E$ be a finite-dimensional vector space, and let $(g_{n})_{n \in
\mathbb N}$ be a sequence of elements of $\operatorname{GL}(E)$.  This
sequence induces a sequence $(\bar{g}_{n})_{n \in \mathbb N}$ of
projective transformations of $\Pr(E)$.  Let $\|\cdot\|$ be an auxiliary
Euclidean norm on $E$ and let $\|\cdot \|_{\infty}$ be the associated operator
norm on the space of endomorphisms $\operatorname{End}(E)$.  The
division of $g_{n}$ by its norm $\| g_{n} \|_{\infty}$ does not modify
the projective transformation $\bar{g}_{n}$. Hence we can assume that
$g_{n}$ belongs to the $\|\cdot\|_{\infty}$-unit sphere of
$\operatorname{End}(E)$. This sphere is compact, so $({g}_{n})_{n \in
\mathbb N}$ admits accumulation points. Up to a subsequence, we can
assume that $({g}_{n})_{n \in \mathbb N}$ converges to an element
$g_{\infty}$ of the $\|\cdot\|_{\infty}$-unit sphere. Let $I$ be the image of
$g_{\infty}$, and let $L$ be the kernel of $g_{\infty}$. Let
\begin{equation*}
\bar{g}_{\infty}: \Pr(E) \setminus \Pr(L) \to \Pr(I) \subset \Pr(E) 
\end{equation*}
be the induced map. 

\begin{proposition}
\label{pro.projective limit}
For any compact $K \subset \Pr(E) \setminus \Pr(L)$, the
restriction of the sequence $(\bar{g}_{n})_{(n \in \mathbb N)}$ on $K$ converges uniformly to
the restriction on $K$ of $\bar{g}_{\infty}$.
\end{proposition}

\begin{corollary}
\label{cor.domaine propre}
Let $\Gamma$ be a \emph{discrete} subgroup of
$\operatorname{PGL}(E)$. Let $\Omega$ be the open subset of $\Pr(E)$
formed by points admitting a neighborhood $U$ such that, for any
sequence $(g_n)$ in $\Gamma$ with accumulation point $g_{\infty}$
having image $I$ and kernel $L$,
$$ U \cap \Pr(L) = U \cap \Pr(I) = \emptyset .$$

Then $\Gamma$ acts properly discontinuously on $\Omega$.
\end{corollary}
\noindent
In fact, the condition $U \cap \Pr(L) = \emptyset$ is sufficient to
define $\Omega$ (as is $U \cap \Pr(I) = \emptyset$).  To see this,
note that if $g_n \to \infty$ with 
\begin{align*}
g_{n}/\|g_{n}\|_{\infty} & \longrightarrow g_{\infty} \\
g_{n}^{-1}/\|g_{n}^{-1}\|_{\infty} & \longrightarrow g^{-}_{\infty}, 
\end{align*}
then 
\begin{equation*}
g_\infty \circ g^-_\infty = g^-_\infty \circ g_\infty = 0.  
\end{equation*}
Hence 
\begin{equation*}
\operatorname{Im} (g_{\infty}) \subseteq
\operatorname{Ker} (g^{-}_{\infty}) 
\text{~and~}
\operatorname{Im}(g^{-}_{\infty}) \subseteq
\operatorname{Ker}(g_{\infty}).
\end{equation*}

\subsection{Cartan's decomposition $G=KAK$}  
\label{sub.kak}\index{Cartan decomposition}\index{polar decomposition}
When $(g_{n})_{n \in \mathbb N}$ is a sequence in a semisimple Lie
group $\operatorname{G}$, a very convenient way to identify the
accumulation points $\bar{g}_{\infty}$ is to use the        
$\operatorname{KAK}$-decomposition in $\operatorname{G}$: first select
the Euclidean norm $\|\cdot\|$ so that it is preserved by the maximal compact
subgroup $\operatorname{K}$ of $\operatorname{G}$. Decompose every
$g_{n}$ in the form $k_{n}a_{n}k'_{n}$, where $k_{n}$ and $k'_{n}$
belong to $\operatorname{K}$, and $a_{n}$ belongs to a fixed Cartan
subgroup.  We can furthermore require that $a_{n}$ is the image by the
exponential of an element of the closure of a Weyl chamber. Up to a
subsequence, $k_{n}$ and $k'_{n}$ admit limits $k_{\infty}$ and
$k'_{\infty}$, respectively.  Composition on the right or on the left
by an element of $\operatorname{K}$ does not change the operator norm,
so $g_{n}$ has $\|\cdot\|_{\infty}$-norm $1$ if and only if $a_{n}$ has
$\|\cdot\|_{\infty}$-norm $1$.  Let ${a}_{\infty}$ be an accumulation point of
$(a_{n})_{n \in \mathbb N}$.  Then 
\begin{equation*}
g_{\infty} = k_{\infty}a_{\infty}k'_{\infty}.   
\end{equation*}
The kernel of $g_{\infty}$ is the
image by $(k'_{\infty})^{-1}$ of the kernel of $a_{\infty}$, and the
image of $g_{\infty}$ is the image by $k_{\infty}$ of the image of
$a_{\infty}$.  Hence, in order to find the singular projective limit
$\bar{g}_{\infty}$, the main task is to find the limit $a_{\infty}$,
and this problem is particularly easy when the rank of
$\operatorname{G}$ is small.

\subsubsection{Sequences in $\Sp(4,\RR)$}
The image by the exponential map of a Weyl chamber in $\sp(4, \RR)$ is
the semigroup $A^{+} \subset A$ of matrices (see \S\ref{sub.weylgroup}):

\begin{equation*}
A(\alpha_{1}, \alpha_{2}) = 
\left[
\begin{array}{cccc}
\exp(\alpha_{1}) & & &  \\
  & \exp(-\alpha_{1}) & & \\
  &   &  \exp(\alpha_{2}) &   \\
  &   &  &   \exp(-\alpha_{2})
\end{array}
\right]
\qquad \alpha_{2} > \alpha_{1} > 0 .
\end{equation*}
The operator norm of $A(\alpha_{1}, \alpha_{2})$ is $\exp(\alpha_{2})$.
We therefore can distinguish three kinds of dynamical 
behaviour for a sequence 
$(A(\alpha_{1}^{(n)}, \alpha_{2}^{(n)}))_{n \in \mathbb N}$:

\begin{itemize}
\item 
\emph{no distortion:\/} when $\alpha^{(n)}_{1}$ and
$\alpha^{(n)}_{2}$ remain bounded, \index{sequence!no-distortion}
\item 
\emph{bounded distortion:\/} when $\alpha^{(n)}_{1}$ and
$\alpha^{(n)}_{2}$ are unbounded, but the difference $\alpha^{(n)}_{2}
- \alpha^{(n)}_{1}$ is bounded, \index{sequence!bounded distortion}
\item 
\emph{unbounded distortion:\/} when the sequences $\alpha^{(n)}_{1}$
and $\alpha^{(n)}_{2} - \alpha^{(n)}_{1}$ are unbounded.
\index{sequence!unbounded}
\end{itemize}
This 
distinction
extends to any sequence $(g_{n})_{n \in \mathbb N}$
in $\Sp(4,\RR)$. Assume that the sequence
$(g_{n}/\|g_{n}\|_{\infty})_{n \in \mathbb N}$ converges to a limit
$g_{\infty}$. Then: 

\begin{itemize}
\item For no distortion, the limit $g_{\infty}$ is not
singular---the sequence $(g_{n})_{n \in \mathbb N}$ converges in
$\Sp(4,\RR)$.
\item For bounded distortion, the kernel $L$ and the image $I$
are $2$-dimensional. More precisely, they are Lagrangian subspaces of
$V$. The singular projective transformation $\bar{g}_{\infty}$ is
defined in the complement of a projective line and takes values in a
projective line; these projective lines are both tangent everywhere to
the contact structure. 
\item For unbounded distortion, the singular projective transformation
$\bar{g}_{\infty}$ is defined in the complement of a projective
hyperplane and admits only one value.
\end{itemize}

\subsubsection{Sequences in $\SOh^{+}(3,2)$}
The Weyl chamber of $\SOh^{+}(3,2)$ is simply the image of the Weyl
chamber of $\sp(4, \RR)$ by the differential of the homomorphism
\begin{equation*}
\Sp(4,\RR) \to \SOh^{+}(3,2)  
\end{equation*}
defined in \S{6.2}.  More precisely,
the image of an element $A(\alpha_{1},
\alpha_{2})$ of $A^{+}$ is $A'(a_{1}, a_{2})$ where
\begin{equation*}
a_{1} = \alpha_{1} + \alpha_{2}, a_{2} = \alpha_{2} - \alpha_{1} 
\end{equation*}
and:
\begin{equation*}
A'(a_{1}, a_{2}) = 
\left[
\begin{array}{ccccc}
\exp(a_{1}) & & & &  \\
  & \exp(a_{2}) & & &  \\
  & & 1 & & \\
  &   &  & \exp(-a_{2}) &   \\
  &  & &  &   \exp(-a_{1})
\end{array}
\right]
\qquad a_{1} > a_{2} > 0 .
\end{equation*}

The $KAK$ decomposition of $\Sp(4,\RR)$ above corresponds under the homomorphism to a $KAK$ decomposition of $\SOh^+(3,2)$.  Reasoning as in the previous section, we distinguish three cases:
\begin{itemize}
\item 
\emph{no distortion:\/} when $a^{(n)}_{1}$ and $a^{(n)}_{2}$
remain bounded,
\item 
\emph{balanced distortion:\/} when $a^{(n)}_{1}$ and
$a^{(n)}_{2}$ are unbounded, but the difference $a^{(n)}_{1} -
a^{(n)}_{2}$ is bounded, \index{sequence!balanced distortion}
\item 
\emph{unbalanced distortion:\/} 
when the sequences $a^{(n)}_{1}$ and $a^{(n)}_{1} - a^{(n)}_{2}$
are unbounded.
\index{sequence!unbalanced}
\end{itemize}
The dynamical analysis is similar, but we restrict to the closed
subset $\EinTO$ of $\Pr(\RR^{3,2})$:

\begin{itemize}
\item 
No distortion corresponds to sequences $(g_{n})_{n \in \mathbb N}$
converging in $\SOh^{+}(3,2)$.
\item For balanced distortion,  the intersection between 
$\Pr(L)$ and $\EinTO$, and the intersection
between $\Pr(I)$ and $\EinTO$ are both photons. 
Hence the restriction of the singular projective transformation
$\bar{g}_{\infty}$ to $\EinTO$ is defined in the complement of a photon and 
takes value in a photon.
\item For unbalanced distortion, the singular projective transformation
$\bar{g}_{\infty}$ is defined in the complement of a lightcone and admits only 
one value.
\end{itemize}

\subsection{Maximal domains of properness}

Most of the time, applying directly Proposition~\ref{pro.projective
limit} and Corollary~\ref{cor.domaine propre} to a discrete subgroup
$\Gamma$ of $\Sp(4,\RR)$ or $\SOh^{+}(3,2)$ in order to find domains
where the action of $\Gamma$ is proper is far from optimal.

Through the morphism $\Sp(4,\RR) \to \SOh^{+}(3,2)$, a sequence in
$\Sp(4,\RR)$ can also be considered as a sequence in
$\SOh(3,2)$. Observe that our terminology is coherent: a sequence has
no distortion in $\Sp(4,\RR)$ if and only if it has no distortion in
$\SOh^{+}(3,2)$. Observe also that since 
\begin{align*}
a_{1} & = \alpha_{1} + \alpha_{2}, \\
a_{2} & = \alpha_{2} - \alpha_{1}, 
\end{align*}
a sequence with bounded distortion in $\Sp(4,\RR)$ is unbalanced in
$\SOh^{+}(3,2)$, and a sequence with balanced distortion in
$\SOh^{+}(3,2)$ is unbounded in $\Sp(4,\RR)$.  In summary, we
distinguish three different kinds of non-converging dynamics, covering
all the possibilities:

\begin{definition}
A sequence $(g_{n})_{n \in \mathbb N}$ of elements of $\Sp(4,\RR)$
escaping from any compact subset  
in $\Sp(4,\RR)$ has:
\begin{itemize}
\item \emph{bounded distortion\/} if the coefficient $a_{2}^{(n)} =
\alpha_{2}^{(n)} - \alpha^{(n)}_{1}$ is bounded,
\item \emph{balanced distortion\/} if the coefficient
$\alpha_{2}^{(n)} = (a_{1}^{(n)} + a^{(n)}_{2})/2$ is bounded,
\item \emph{mixed distortion\/} if all the coefficients $a_{1}^{(n)}$,
$a_{2}^{(n)}$, $\alpha_{1}^{(n)}$, $\alpha_{2}^{(n)}$ are unbounded.
\index{sequence!mixed distortion}
\end{itemize}
\end{definition}

\subsubsection{Action on $\EinTO$}
The dynamical analysis can be refined in the mixed distortion case.
In \cite{Frances2}, C.\ Frances proved: 

\begin{proposition}
\label{pro.refineddynamic}
Let $(g_{n})_{n \in \mathbb N}$ be a sequence of elements of
$\SOh^{+}(3,2)$ with mixed distortion, such that the sequence
$(g_{n}/\| g_{n} \|_{\infty})_{n \in \mathbb N}$ converges to an
endomorphism $g_{\infty}$.  Then there are photons $\Delta^{-}$ and
$\Delta^{+}$ in $\EinTO$ such that, for any sequence $(p_{n})_{n \in
\mathbb N}$ in $\EinTO$ converging to an element of $\EinTO \setminus
\Delta^{-}$, all the accumulation points of $(g_{n}(p_{n}))_{n \in
\mathbb N}$ belong to $\Delta^{+}$.
\end{proposition}
\noindent
As a corollary (\S{4.1} in \cite{Frances2}):

\begin{corollary}
\label{cor.domaine propre2}
Let $\Gamma$ be a \emph{discrete} subgroup of $\SOh^{+}(3,2)$. Let
$\Omega_{0}$ be the union of all open domains $U$ in $\EinTO$ such
that, for any accumulation point $g_{\infty}$, with kernel $L$ and image $I$, 
of a sequence
$(g_{n}/\|g_{n}\|_{\infty})_{n \in \mathbb N}$ with $g_{n} \in
\SOh^{+}(3,2)$:

\begin{itemize}

\item 
When $(g_{n})_{n \in \mathbb N}$ has balanced distortion, $U$ is
disjoint from the photons $\Pr(L) \cap \EinTO$ and $\Pr(I) \cap
\EinTO$;
 
\item 
When $(g_{n})_{n \in \mathbb N}$ has bounded distortion, $U$ is
disjoint from the lightcone $\Pr(L) \cap \EinTO$;
 
\item 
When $(g_{n})_{n \in \mathbb N}$ has mixed distortion, $U$ is disjoint
from the photons $\Delta_{-}$ and $\Delta_{+}$.

\end{itemize}
Then the action of $\Gamma$ on $\Omega_{0}$ is properly discontinuous.
\end{corollary}

\noindent
Observe that the domain $\Omega_{0}$ is in general bigger than the
domain $\Omega$ appearing in Corollary~\ref{cor.domaine propre}.  An
interesting case is that in which $\Omega_{0}$ is obtained by
removing only photons:

\begin{proposition}[Frances~\cite{Frances2}]
\label{pro.frances4}
A discrete subgroup $\Gamma$ of $\SOh^{+}(3,2)$ 
does not contain sequences with bounded distortion
if and only if its action on $\Pr(\RR^{3,2}) \setminus \EinTO$ is
properly discontinuous.
\end{proposition}
\noindent
Frances calls such a subgroup a {\em of the first kind.\/}
The following suggests that the domain $\Omega_{0}$ is
optimal. 

\begin{proposition}[Frances~\cite{Frances2}]
\label{pro.frances6}
Let $\Gamma$ be a discrete, Zariski dense subgroup of $\SOh^{+}(3,2)$
which does not contain sequences with bounded distortion.
Then $\Omega_{0}$ is the unique maximal open subset of $\EinTO$ 
on which $\Gamma$ acts properly.
\end{proposition}
\index{discrete subgroup!of the first kind} 
%
%
%
%
%
%
%
%

\subsubsection{Action on $\Pr(V)$}

A similar analysis should be done when $\Gamma$ is considered a
discrete subgroup of $\Sp(4,\RR)$ instead of $\SOh^+(3,2)$. 
The following proposition is analogous to Proposition~\ref{pro.refineddynamic}:

\begin{proposition}
\label{pro.refineddynamic2}
Let $(g_{n})_{n \in \mathbb N}$ be a sequence of elements of $\Sp(V)$
with mixed distortion, such that the sequence $(g_{n}/\| g_{n}
\|_{\infty})_{n \in \mathbb N}$ converges to an endomorphism
$g_{\infty}$ of $V$.  Then there are contact projective lines
$\Delta^{-}$ and $\Delta^{+}$ in $\Pr(V)$ such that, for any sequence
\begin{equation*}
(p_{n})_{n \in \mathbb N}\in \Pr(V) 
\end{equation*}
 converging to an element of
$\Pr(V) \setminus \Delta^{-}$, 
all the accumulation points of
$(g_{n}(p_{n}))_{n \in \mathbb N}$ belong to $\Delta^{+}$.
\end{proposition}

We can then define a subset $\Omega_1$ of $\Pr(V)$ as the interior of
the subset obtained after removing limit contact projective lines
associated to subsequences of $\Gamma$ with bounded or mixed
distortion, and removing projective hyperplanes associated to
subsequences with balanced distortion. Then it is easy to prove that
the action of $\Gamma$ on $\Omega_1$ is properly discontinuous.

An interesting case is that in which we remove only projective lines,
and no hypersurfaces---the case in which $\Gamma$ has no subsequence with
balanced distortion. Frances calls such $\Gamma$ {\em groups of the
second kind.\/} The following questions arise from comparison with
Propositions \ref{pro.frances6} and \ref{pro.frances4}:

\index{discrete subgroup!of the second kind}
\medskip
\noindent \emph{Question:\/} Can groups of the second kind be defined
as groups acting properly on some associated space?

\medskip
\noindent \emph{Question:\/} Is $\Omega_{1}$ the unique maximal open
subset of $\Pr(V)$ on which the action of $\Gamma$ is proper, at least
if $\Gamma$ is Zariski dense?

\subsubsection{Action on the flag manifold}
Now consider the action of $\Sp(4,\RR)$ on the flag manifold $\Flag$.
Let $v,w \in V$ be such that $\omega(v,w) = 0$, so $v$ and $w$ span a Lagrangian plane.  Let 
\begin{align*}
\Flag &\xrightarrow{\rho_1} \Pho \\
\Flag &\xrightarrow{\rho_2} \EinTO
\end{align*}
be the natural projections.  Let $g_n$ be a sequence in
$\Sp(4,\RR)$ diverging to infinity with mixed distortion.
We invite the reader to verify the following statements:

\begin{itemize}
\item{ There are a flag $q^+ \in \Flag$ and points
$[v] \in \Pr(V)$ and $z \in \EinTO$ such that, on the complement of
\begin{equation*}
\rho_1^{-1}([v^\perp]) \cup \rho_2^{-1}(L(z))  
\end{equation*}
the sequence $g_n$ converges
uniformly to the constant map $q^+$.}

\item{There are contact projective lines
$\alpha^+$,$\alpha^-$ in $\Pr(V)$ and photons $\beta^+$,$\beta^-$ in
$\EinTO$ such that, on the complement of 
\begin{equation*}
\rho_1^{-1}(\alpha^-) \cup \rho_2^{-1}(\beta^-)
\end{equation*}
all accumulation points of $g_n$ lie in
\begin{equation*}
\rho_1^{-1}(\alpha^+) \cap \rho_2^{-1}(\beta^+) .
\end{equation*}
This intersection is homeomorphic to a wedge of two circles.}
\end{itemize}

\section{Crooked surfaces}\label{sec:crooked}
\label{sec.crooked}

Crooked planes were introduced by Drumm~\cite{drumm1,drumm2,drummgoldman} to investigate discrete groups of Lorentzian transformations which act freely and properly on $\Eto$. He used crooked planes to construct
fundamental polyhedra for such actions; they play a role analogous to equidistant surfaces bounding Dirichlet fundamental domains in Hadamard manifolds.  This section discusses the conformal compactification of a crooked plane and its automorphisms.  \index{crooked!plane}

\subsection{Crooked planes in Minkowski space}
For a detailed description of crooked planes, see
Drumm-Goldman~\cite{drummgoldman}. We quickly summarize the basic
results here.

Consider $\Eto$ with the Lorentz metric from the inner product $I_2 \oplus - I_1$ on $\RR^{2,1}$. 
A crooked plane $C$ is a surface in $\Eto$ that divides $\Eto$ into
two cells, called {\em crooked half-spaces.\/} It is a piecewise linear
surface composed of four $2$-dimensional faces, joined along four rays,
which all meet at a point $p$, called the {\em vertex.\/}
The four rays have endpoint $p$, and form two lightlike geodesics, 
which we denote $\ell_1$ and $\ell_2$. Two of the faces are null half-planes
$\Ww_1$ and $\Ww_2$, bounded by $\ell_1$ and $\ell_2$ respectively, 
which we call
 {\em wings.\/} \index{wing}
The two remaining faces consist of the intersection
between $J^\pm(p)$ and the timelike plane $P$ containing $\ell_1$ and $\ell_2$; their union is the {\em stem\/} of $C$.\index{stem} 
The timelike plane $P$ is the orthogonal complement of a unique spacelike line
$P^\perp(p)$ containing $p$, called the {\em spine\/} of $C$.

To define a crooked plane, we first define the wings, stem, and spine.  A lightlike geodesic $\ell = p + \RR v$ lies in a unique null plane $\ell^\perp$
(\S \ref{sub.Minkowski space}). The ambient orientation of $\RR^{2,1}$
distinguishes a component of $\ell^\perp\setminus\ell$ as follows.
Let $u\in \RR^{2,1}$ be a timelike vector such that $\langle u,v \rangle < 0$.
Then each component of $\ell^\perp\setminus\ell$ defined by
\begin{align*}
\Ww^+(\ell) & := \left\{ p + w \in \ell^\perp \mid \det(u,v,w) > 0 \right\}  \\
\Ww^-(\ell) & := \left\{ p + w \in \ell^\perp \mid \det(u,v,w) < 0 \right\}
\end{align*}
is   independent  of the choices  above.  In  particular, every
orientation-preserving isometry $f$ of $\Eto$ maps 
\begin{align*}
\Ww^+(\ell)&\longrightarrow\Ww^+(f(\ell))\\
\Ww^-(\ell)& \longrightarrow \Ww^-(f(\ell))
\end{align*}
and every orientation-reversing isometry $f$ maps
\begin{align*}
\Ww^+(\ell)&\longrightarrow\Ww^-(f(\ell))\\
\Ww^-(\ell)& \longrightarrow \Ww^+(f(\ell)) .
\end{align*}
Given two lightlike geodesics $\ell_1,\ell_2$ containing $p$, the
stem is defined as 
\begin{equation*}
\Ss(\ell_1,\ell_2) := \JJ^{\pm}(p) \cap (p + \operatorname{span}\{ \ell_1 - p,\ell_2- p \}).
\end{equation*}
The spine is
\begin{equation*}
\sigma = p + (\Ss(\ell_1,\ell_2) - p)^\perp .
\end{equation*}
Compare Drumm-Goldman~\cite{drummgoldman}.

The {\em positively-oriented crooked plane\/} with vertex $p$ and stem
$\Ss( \ell_1 , \ell_2 )$ is the union
\index{crooked!plane!positively-oriented}
\begin{equation*}
\Ww^+(\ell_1) \cup \Ss(\ell_1,\ell_2) \cup  \Ww^+(\ell_2) .
\end{equation*}
Similarly, the {\em negatively-oriented crooked plane\/} with vertex
$p$ and stem $\Ss( \ell_1, \ell_2)$ is
\index{crooked!plane!negatively-oriented}
\begin{equation*}
\Ww^-(\ell_1) \cup \Ss(\ell_1,\ell_2) \cup  \Ww^-(\ell_2) .
\end{equation*}
Given an orientation on $\Eto$, a positively-oriented crooked plane is determined by its vertex and its spine.  Conversely, every point $p$ and spacelike line $\sigma$ containing $p$
determines a unique positively- or negatively-oriented crooked plane.

A crooked plane $C$ is homeomorphic to $\RR^2$, and the complement
$\Eto \setminus C$ consists of two components, each homeomorphic
to $\RR^3$.  The components of the complement of a crooked plane are called
{\em open crooked half-spaces\/} and their closures {\em closed crooked half-spaces.\/}
The spine of $C$ is the unique spacelike line contained in $C$.

\subsection{An example}\label{sec:crookedexample}
Here is an example of a crooked plane with vertex the origin and
spine the $x$-axis:
\begin{equation*}
p = \bmatrix 0 \\ 0 \\ 0 \endbmatrix, \ 
\sigma = \RR  \bmatrix 1 \\ 0 \\ 0 \endbmatrix .
\end{equation*}
The lightlike geodesics are
\begin{equation*}
\ell_1 = \RR  \bmatrix 0 \\ -1 \\ 1 \endbmatrix, \qquad
\ell_2 = \RR  \bmatrix 0 \\ 1 \\ 1 \endbmatrix ,
\end{equation*}
the stem is
\begin{equation*}
\left\{ \bmatrix 0 \\ y \\ z \endbmatrix \ : \  
y^2 - z^2 \le 0  \right\}
\end{equation*}
and the wings are
\begin{align*}
\Ww_1 & =  
\left\{ \bmatrix x \\ y \\ -y \endbmatrix \ : \  
x \ge 0, y\in\RR \right\} \\
\Ww_2 & =  
\left\{ \bmatrix x \\ y \\ y \endbmatrix \ : \  
x \le 0, y\in\RR \right\} .
\end{align*}
The identity component of $\operatorname{Isom}(\Eto)$ acts 
transitively on the space of pairs of vertices and unit spacelike vectors, 
so it is transitive on positively-oriented and negatively-oriented crooked planes.
An orientation-reversing isometry exchanges positively- and negatively-oriented crooked planes, so $\operatorname{Isom}(\Eto)$ acts transitively on the set of all crooked planes.

\subsection{Topology of a crooked surface}
The closures of crooked planes in Minkowski patches are 
{\em crooked surfaces.\/} These were studied in Frances~\cite{Frances1}.
In this section we describe the topology of a crooked surface.
\index{crooked!surface}

Let $C\subset \Eto$ be a crooked plane. 
\begin{theorem} 
The closure $\overline{C}\in\EinTO$ is a topological submanifold
homeomorphic to a Klein bottle. The lift of
$\overline{C}$ to the double covering $\uEinTO$ is the oriented double
covering of $\overline{C}$ and is homeomorphic to a torus.
\end{theorem}

\begin{proof}
Since the isometry group of Minkowski space acts transitively
on crooked planes, it suffices to consider the single crooked plane $C$
defined in \S\ref{sec:crookedexample}.

Recall the stratification of $\EinTO$ from \S \ref{sub.confeinstein}.
Write the nullcone $\Ntt$ of $\RR^{3,2}$ as
\begin{equation*}
\bmatrix X \\ Y \\ Z \\ U \\ V \endbmatrix \text{~where~}
X^2 + Y^2 - Z^2 - U V = 0. 
\end{equation*}
The homogeneous coordinates of points in the stem $\Ss(C)$ satisfy
\begin{equation*}
X = 0, \qquad  Y^2 - Z^2 \le 0, \qquad V \neq 0
\end{equation*}
and thus the closure of the stem $\overline{\Ss(C)}$ is defined by
(homogeneous) inequalities
\begin{equation*}
X = 0, \qquad  Y^2 - Z^2 \le 0.
\end{equation*}
The two lightlike geodesics 
\begin{equation*}
\ell_1 = \RR\bmatrix 0  \\ -1 \\ 1 \endbmatrix, \qquad
\ell_2 = \RR\bmatrix 0  \\ 1 \\ 1 \endbmatrix
\end{equation*}
defining $\Ss(C)$ extend to photons
$\phi_1, \phi_2$ with ideal points represented in homogeneous coordinates
\begin{equation*}
p_1 = \bmatrix 0 \\ -1 \\ 1 \\ 0 \\ 0 \endbmatrix, \qquad
p_2 = \bmatrix 0 \\ 1 \\ 1 \\ 0 \\ 0 \endbmatrix .
\end{equation*}
The closures of the corresponding wings $\Ww_1,\Ww_2$ are described in 
homogeneous coordinates by:
\begin{align*}
\overline{\Ww_1} &= \left\{ \bmatrix X \\ -Y \\ Y \\ U \\ V \endbmatrix
\;:\; X^2 - U V = 0, \quad X V \ge 0 \right\} \\
\overline{\Ww_2} &= \left\{ \bmatrix X \\ Y \\ Y \\ U \\ V \endbmatrix
\;:\; X^2 - U V = 0, \quad X V \le 0 \right\} .
\end{align*}
The closure of each wing intersects the ideal lightcone 
$L(p_\infty)$ (described by $V=0$) in the photons:
\begin{align*}
\psi_1 &= \left\{ \bmatrix 0 \\ -Y \\ Y \\ U \\ 0 \endbmatrix
\;:\; Y,U\in\RR \right\} \\
\psi_2 &= \left\{ \bmatrix 0 \\ Y \\ Y \\ U \\ 0 \endbmatrix
\;:\; Y,U\in\RR \right\} .
\end{align*}
Thus the crooked surface $\overline{C}$ decomposes into the following strata:
\begin{itemize}
\item four points in a stem configuration: the vertex $p_0$, the improper point $p_\infty$,
and the two ideal points $p_1$ and $p_2$;
\index{stem configuration}
\item eight line segments, the components of
\begin{align*}
\phi_1 & \setminus \{p_0,p_1\}  \\
\phi_2 & \setminus \{p_0,p_2\}  \\
\psi_1 & \setminus \{p_\infty,p_1\}  \\
\psi_2 & \setminus \{p_\infty,p_2\} ;
\end{align*}
\item two null-half planes, the interiors of the wings $\Ww_1,\Ww_2$;
\item the two components of the interior of the stem $\Ss$.
\end{itemize}
Recall that the inversion in the unit  sphere $\iota=\Id_3 \oplus \bmatrix 0 & 1 \\ 1 & 0 \endbmatrix$ fixes $p_1$ and $p_2$, and interchanges
$p_0$ and $p_\infty$. Moreover $\iota$ interchanges
$\phi_i$ with $\psi_i$, $i = 1,2$.
Finally $\iota$ leaves invariant the interior of each $\Ww_i$ 
and interchanges the two components of the interior of $\Ss$.

The original crooked plane equals 
\begin{equation*}
\{p_0\} \;\cup\; \phi_1\setminus\{p_1\}\;\cup\; \phi_2\setminus\{p_2\} \;\cup\;
\operatorname{int}(\Ww_1) \;\cup\; \operatorname{int}(\Ww_2) 
\;\cup\; \operatorname{int}(\Ss)
\end{equation*}
and is homeomorphic to $\RR^2$.  The homeomorphism is depicted
schematically in Figure~\ref{fig:crookedvertex}.  The interiors of
$\Ww_1,\Ww_2$, and $\Ss$ correspond to the four quadrants in $\RR^2$.  The
wing $\Ww_i$ is bounded by the two segments of $\phi_i$, whereas each
component of $\Ss$ is bounded by one segment of $\phi_1$ and one
segment of $\phi_2$. These four segments correspond to the four
coordinate rays in $\RR^2$. 

Now we can see that $C$ is a topological
manifold: points in
$\operatorname{int}(\Ww_1),\operatorname{int}(\Ww_2),$ or
$\operatorname{int}(\Ss)$ have coordinate neighborhoods in these
faces.  Interior points of the segments have two half-disc
neighborhoods, one from a wing and one from the stem. The vertex $p_0$
has four quarter-disc neighborhoods, one from each wing, and one from
each component of the stem. (See Figure~\ref{fig:crookedvertex}.)

\begin{figure}[ht]
\begin{center}
\includegraphics[width=3cm, height=3cm]{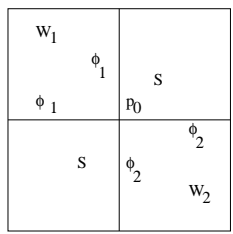}
\end{center}
\caption{Flattening a crooked plane around its vertex} \label{fig:crookedvertex}
\end{figure}

Coordinate charts for the improper point $p_\infty$ and
points in $\psi_i \setminus \{p_\infty,p_i\}$
are obtained by composing the above charts with the inversion $\iota$.
It remains to find coordinate charts near the ideal points $p_1,p_2$.
Consider first the case of $p_1$. The linear functionals on $\Rtt$
defined by
\begin{align*}
T & = Y - Z \\ W & = Y + Z  
\end{align*}
are null since the defining quadratic form factors:
\begin{equation*}
X^2 + Y^2 - Z^2 - U V = X^2 + T W - U V.
\end{equation*}
Working in the affine patch defined by $T \neq 0$ with inhomogeneous
coordinates
\begin{align*}
\xi & := \frac{X}{T} \\ 
\eta & := \frac{Y}{T} \\ 
\omega & := \frac{W}{T} \\ 
\upsilon & := \frac{U}{T} \\ 
\nu & := \frac{V}{T} 
\end{align*}
the nullcone is defined by:
\begin{equation*}
\xi^2 + \omega - \upsilon\nu = 0 
\end{equation*}
whence
\begin{equation*}
\omega = -\xi^2 +  \upsilon\nu
\end{equation*}
and $(\xi,\upsilon,\nu)\in\RR^3$ is a coordinate chart for this 
patch on $\EinTO$. 

In these coordinates, $p_1$ is the origin $(0,0,0)$, $\phi_1$ is the line
$\xi = \upsilon=0$, and $\psi_1$ is the line $\xi = \nu = 0$.
The wing $\Ww_2$ misses this patch, but both $\Ss$ and $\Ww_1$ intersect it.
In these coordinates $\Ss$ is defined by 
\begin{equation*}
\xi = 0, \quad \omega \le 0
\end{equation*}
and $\Ww_1$ is defined by 
\begin{equation*}
\xi \le 0, \quad \omega = 0 .
\end{equation*}
Since on $\Ww_1$
\begin{equation*}
\upsilon\nu = \xi^2 \ge 0
\end{equation*}
this portion of $\Ww_1$ in this patch has two components
\begin{align*}
\upsilon,\nu &< 0  \\
\upsilon,\nu &> 0 
\end{align*}
and the projection $(\upsilon,\nu)$ defines a coordinate chart
for a neighborhood of $p_1$. (Compare Figure~\ref{fig:crookedideal}.)

\begin{figure}[ht]
\begin{center}
\includegraphics[width=3cm, height=3cm]{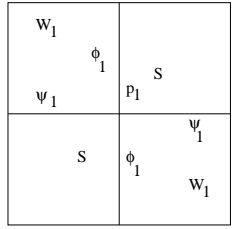}
\end{center}
\caption{Flattening a crooked surface around an ideal point $p_1$}  \label{fig:crookedideal}
\end{figure}

The case of $p_2$ is completely analogous.  It follows that $\overline{C}$ is a closed surface with cell decomposition
with four $0$-cells, eight $1$-cells and four $2$-cells. Therefore
\begin{equation*}
\chi(\overline{C}) = 4 - 8 + 4 = 0 
\end{equation*}
and $\overline{C}$ is homeomorphic to either a torus or a Klein bottle.

To see that $\overline{C}$ is nonorientable, consider a photon, for example
$\phi_1$. Parallel translate the null geodesic
$\phi_1\setminus\{p_1\}$ to a null geodesic $\ell$ lying on the wing
$\Ww_1$ and disjoint from $\phi_1\setminus\{p_1\}$. Its closure
$\bar{\ell} = \ell \cup \{p_1\}$ is a photon on
$\overline{\Ww}_1 \subset \overline{C}$ which intersects $\phi_1$
transversely with intersection number $1$. Thus the self-intersection number
\begin{equation*}
\phi_1 \cdot  \phi_1  = 1
\end{equation*}
so $\phi_1\subset \overline{C}$ is an orientation-reversing loop.
Thus $\overline{C}$ is nonorientable, and homeomorphic to a Klein bottle.
\end{proof}


Next we describe the stratification of a crooked surface in the double covering $\uEinTO$.  Recall from \S\ref{sec:improperpatch} that a Minkowski patch in $\uEinTO$ has both a spatial and a timelike improper point.  Let $C$ be a crooked plane of $\Eto$, embedded in a Minkowski patch $\Min^+(p_\infty)$, so $p_\infty = p^{\mathrm{ti}}_\infty$, the timelike improper point of this patch.  Denote by $p^{\mathrm{sp}}_\infty$ the spatial improper point.

The closure $\overline{C}$ of $C$ in $\uEinTO$ decomposes into the following strata:
\begin{itemize}
\item seven points: $p_0, p_\infty^{\mathrm{ti}}, p_\infty^{\mathrm{sp}}, p_1^\pm,p_2^\pm$;

\item twelve photon segments: 
\begin{align*}
\phi_i^\pm,  \   \mbox{connecting $p_0$  to $p_i^\pm$} \\
\alpha_i^\pm,  \  \mbox{connecting $p_\infty^{\mathrm{ti}}$ to $p_i^\pm$} \\
\beta_i^\pm, \   \mbox{connecting $p_\infty^{\mathfrak{sp}}$ to $p_i^\pm$} ;
\end{align*}

\item two null half-planes, the interiors of $\Ww_1$ and $\Ww_2$.  The wing $\Ww_i$ is bounded by the curves $\phi_i^\pm$ and $\beta_i^\pm$;

\item the two components of the interior of the stem $\Ss$.  The stem is bounded by the curves $\phi_i^\pm$ and $\alpha_i^\pm$, for $i=1,2$.
\end{itemize}
The saturation of $\overline{C}$ by the antipodal map on $\uEinTO$ is the lift of a crooked surface from $\EinTO$.  The interested reader can verify that it is homeomorphic to a torus.

\subsection{Automorphisms of a crooked surface}
\label{sub.crookedauto}
Let $C$ be the positively-oriented crooked plane of Section \ref{sec:crookedexample}, and $\overline{C}$ the associated crooked surface in $\EinTO$.  First, $C$ is invariant by all positive homotheties centered at the origin, because each of the wings and the stem are.  Second, it is invariant by the $1$-dimensional group of linear hyperbolic isometries of Minkowski space preserving the lightlike lines bounding the stem.  The subgroup $A$, which can be viewed as the subgroup of $\SOh(3,2)$ acting by positive homotheties and positive linear hyperbolic isometries of Minkowski space, then preserves $C$, and hence $\overline{C}$.  The element

$$ s_0 = \left( \begin{array}{ccccc}
1 &  &  &  &  \\
  & -1 &  &   & \\
 &   &  -1 &  &  \\
  &    &  & 1 &   \\
 &  &  &  &  1
\end{array}
\right)
$$
is a reflection in the spine, and also preserves $\overline{C}$. Note that $s_0$ is time-reversing.  Then we have 
$$ \ZZ_2 \ltimes A \cong \ZZ_2 \ltimes (\RR_{>0}^*)^2 \subset \mbox{Aut}(\overline{C}) .$$

Next let $\ell_1, \ell_2$ be the two lightlike geodesics bounding the stem (alternatively bounding the wings) of $C$.  As above, the inversion $\iota$ leaves invariant $C \setminus (\ell_1\cup\ell_2 )$.  In fact, the element

$$ s_1 = \left( \begin{array}{ccccc}
-1   &    &   &   &   \\
   & -1 &  &    &   \\
   &    & -1 &  & \\
  &   &  &   & 1 \\
  &  &  &  1 &  
\end{array}
\right)
$$ 
is an automorphism of $\overline{C}$.  The involution
$$ s_2 = \left( \begin{array}{ccccc}
-1 &    &    &    &     \\
   & 1  &    &    &     \\
   &    & -1 &    &     \\
   &    &    & 1  &    \\
   &    &   &     & 1
\end{array}
\right)
$$
also preserves $\overline{C}$ and exchanges the ideal points $p_1$ and $p_2$.  The involutions $s_0, s_1,$ and $s_2$ pairwise commute, and each product is also an involution, so we have 
$$ G := \ZZ_2^3 \ltimes (\RR_{>0}^*)^2 \subset \mbox{Aut}(\overline{C})$$

To any crooked surface can be associated a quadruple of points in a stem configuration. \index{stem configuration} The stabilizer of a stem configuration in $\SOh(3,2) \cong \operatorname{PO}(3,2)$ is $N(A)$, the normalizer of a Cartan subgroup $A$.  Suppose that the points $(p_0,p_1,p_2, p_\infty)$ are associated to $\overline{C}$.  As above, a neighborhood of $p_0$ in $\overline{C}$ is not diffeomorphic to a neighborhood of $p_1$ in $\overline{C}$, so any automorphism must in fact belong to the subgroup $N'(A)$ preserving each pair $\{ p_0,p_\infty \}$ and $\{ p_1,p_2 \}$.  
Each $g \in N'(A)$ either preserves $\overline{C}$ or carries it to its \emph{opposite}, the closure of the negatively-oriented crooked plane having the same vertex and spine as $C$.
Now it is not hard to verify that the full automorphism group of $\overline{C}$ in $\SOh(3,2)$ is $G$.

\section{Construction of discrete groups}
\label{sec.discrete}
A {\em complete flat Lorentzian manifold\/} is a quotient $\Eno/\Gamma$,
where $\Gamma$ acts freely and properly discontinuously 
on $\Eno$.
When $n=2$, Fried and Goldman~\cite{FriedGoldman} 
showed that unless $\Gamma$ is solvable,
projection on $\Oh(2,1)$ is necessarily injective and,
furthermore, this linear part is a discrete subgroup 
$\Gamma_0\subset \Oh(2,1)$\cite{Abels,CDGM,Milnor}.

In this section we identify $\Eto$ with its usual embedding in $\EinTO$, so that we consider such $\Gamma$ as discrete subgroups of $\SOh(3,2)$.  We will look at the resulting actions on Einstein space, as well as on photon space.  At the end of the section, we list some open questions.

\subsection{Spine reflections}
\label{sub.discreteinv}

In \S \ref{sub.crookedauto}, we described the automorphism group of a crooked surface.  We recall some of the basic facts about the reflection in the spine of a crooked surface, which is discussed in \S \ref{sub.conformalinv} and \S \ref{sub.sub.lagrangianpair}, and which is denoted $s_0$ in the example above.  Take the inner product on $\RR^{3,2}$
to be given by the matrix 
$$\Id_2\oplus -\Id_1 \oplus \left(-\frac{1}{2}\right) \begin{bmatrix}0&1\\ 1&0\end{bmatrix}$$
and identify $\Eto$ with its usual embedding in the Minkowski patch determined by the improper point $p_\infty$.  Let $C$ be the crooked plane determined by the stem configuration $(p_0,p_1,p_2,p_\infty)$ as in \S \ref{sec:crookedexample}, with
\begin{equation*}
p_1=\begin{bmatrix}0\\ -1\\ 1\\ 0\\ 0\end{bmatrix} \mbox{~and~}
p_2=\begin{bmatrix}0\\ 1\\ 1\\ 0\\ 0\end{bmatrix}.
\end{equation*}
Then $s_0$
is an orientation-preserving,
time-reversing involution having fixed set
\begin{equation*}
\Fix(s_0) \;=\; 
\{p_1,p_2\}\,\cup\, \big( L(p_1)\cap L(p_2)\big) .
\end{equation*}
In the Minkowski patch,
$\langle s_0 \rangle$ interchanges the two components
of the complement of $C$.  

If a set of crooked planes in $\Eto$ is pairwise disjoint, then the group generated by reflections in their spines acts properly discontinuously on the entire space~\cite{CharetteGoldman,drumm1,drummgoldman}.  Thus spine reflections associated to disjoint crooked planes give rise to discrete subgroups of $\SOh(3,2)$.  We will outline a way to construct such groups; see~\cite{Charette}, for details.

Let $S_1,~S_2\subset\EinTO$ be a pair of spacelike circles that intersect in a point; conjugating if necessary, we may assume that this point is $p_\infty$.  Each circle $S_i$, $i=1,2$, is the projectivized nullcone of a subspace $V_i\subset\RR^{3,2}$ of type (2,1); $V_1+V_2$ can be written as the direct sum
 \begin{equation*}
 \RR v_1\oplus\RR v_2\oplus W,
 \end{equation*}
 where $v_1,~v_2$ are spacelike vectors and $W=V_1\cap V_2$ is of type (1,1).  We call $\{ S_1, S_2\}$ an {\em ultraparallel pair} if $v_1^\perp \cap v_2^\perp$ is spacelike.  Alternatively, we can define the pair to be ultraparallel if they are parallel to vectors $u_1,~u_2\in\Rto$ such that $u_1^\perp\cap u_2^\perp$
 is a spacelike line in $\Eto$.

Let $S_1$, $S_2$ be an ultraparallel pair of spacelike circles in $\EinTO$.  Denote by $\iota_1$ and $\iota_2$  the spine reflections fixing the respective circles.  (Note that $\iota_1$ and $\iota_2$ are conjugate to $s_0$, since $\SOh(3,2)$ acts transitively on crooked surfaces.) Identifying the subgroup of $\SOh(3,2)$ fixing $p_0$ and 
$p_\infty$ with the group of Lorentzian linear 
similarities 
\begin{equation*}
\Sim(\Eto) = \RR_+\cdot\Oh(2,1), 
\end{equation*}
then $\gamma=\iota_2\circ\iota_1$ has hyperbolic linear part---that is, it has three, distinct real eigenvalues.  The proof of this fact and the following proposition may be found, for instance, in~\cite{Charette}.

\begin{proposition}
Let $S_1$ and $S_2$ be an ultraparallel pair of spacelike circles as above.  Then $S_1$ and $S_2$ are the spines of a pair of disjoint crooked planes, bounding a fundamental domain for $\langle\gamma\rangle$ in $\Eto$.
\end{proposition}
Note that while $\langle\gamma\rangle$ acts freely and properly discontinuously on $\Eto$, it fixes $p_\infty$ as well as two points on the ideal circle.

Next, let $S_i$, $i=1,2,3$ be a triple of pairwise ultraparallel spacelike circles, all intersecting in $p_\infty$, and let $\Gamma=\langle\iota_1,\iota_2,\iota_3\rangle$ be the associated group of spine reflections.  Then $\Gamma$ contains an index-two free group generated by hyperbolic isometries of $\Eto$ (see \cite{Charette}).  Conversely, we have the following generalization of a well-known theorem in hyperbolic geometry.

\begin{theorem}\cite{Charette}
Let $\Gamma=\langle\gamma_1,\gamma_2,\gamma_3\mid\gamma_1\gamma_2\gamma_3=Id\rangle$ be a subgroup of isometries of $\Eto$, where each $\gamma$ has hyperbolic linear part and such that their invariant lines are pairwise ultraparallel.  Then there exist spine reflections $\iota_i$, $i=1,2,3$, such that $\gamma_1=\iota_1\iota_2$, $\gamma_2=\iota_2\iota_3$ and $\gamma_3=\iota_3\iota_1$.
\end{theorem}

Note that $\Gamma$ as above is discrete.  Indeed, viewed as a group of affine isometries of $\Eto$, its linear part $G\leq\Oh(2,1)$ acts on the hyperbolic plane and is generated by reflections in three ultraparallel lines.  
As mentioned before, if the spacelike circles are spines of pairwise disjoint crooked planes, then $\Gamma$ acts properly discontinuously on the Minkowski patch.  Applying this strategy, we obtain that the set of all properly discontinuous groups $\Gamma$, with linear part generated by three ultraparallel reflections, is non-empty and open~\cite{Charette}.

Here is an example. For $i=1,2,3$, let $V_i\subset\RR^{3,2}$ be the $(2,1)$-subspace 
\begin{equation*}
V_i=\left\{
\begin{bmatrix} 
au_i+cp_i\\ 
a\langle u_i,p_i\rangle+b+c\langle p_i,p_i\rangle \\
c
\end{bmatrix}
~\mid~a,b,c\in\RR 
\right\}                      ,
\end{equation*}
where 
\begin{align*}
u_1 &= \begin{bmatrix} \sqrt{2} & 0 & 1\end{bmatrix}^\dagger & u_2 &= \begin{bmatrix} -\frac{\sqrt{2}}{2} & \frac{\sqrt{6}}{2} & 1\end{bmatrix}^\dagger  & u_3 &= \begin{bmatrix} -\frac{\sqrt{2}}{2} & -\frac{\sqrt{6}}{2} & 1\end{bmatrix}^\dagger ~\\
p_1 &= \begin{bmatrix} 0 & \sqrt{2}  & 1\end{bmatrix}^\dagger  & p_2 &= \begin{bmatrix} -\frac{\sqrt{6}}{2}  & -\frac{\sqrt{2}}{2} & 1\end{bmatrix}^\dagger & p_3 &= \begin{bmatrix} \frac{\sqrt{6}}{2} & & -\frac{\sqrt{2}}{2}  & 1\end{bmatrix}^\dagger .
\end{align*}
Then the projectivized nullcone of $V_i$ is a spacelike circle---in fact, it corresponds to the spacelike geodesic in $\Eto$ passing through $p_i$ and parallel to $u_i$.  The crooked planes with vertex $p_i$ and spine $p + \RR u_i$, respectively, are pairwise disjoint (one shows this using inequalities found in~\cite{drummgoldman}).

\subsection{Actions on photon space}
\label{sub.phoaction}
Still in the same Minkowski patch as above, let $G$ be a finitely generated discrete subgroup of $\Oh(2,1)$ that is free and {\em purely hyperbolic}---that is, every nontrivial element is hyperbolic.  Considered as a group of isometries of the hyperbolic plane, $G$ is a convex cocompact free group.  By Barbot~\cite{Barbot},

\begin{theorem}\label{thm:omegapd}
Let $\Gamma$ be a subgroup of isometries of $\Eto$ with convex cocompact linear part.  Then there is a pair of non-empty, $\Gamma$-invariant, open, convex sets $\Omega^\pm\subset \Eto$ such that
\begin{itemize}
\item The action of $\Gamma$ on $\Omega^\pm$ is free and proper;
\item The quotient spaces $\Omega^\pm/\Gamma$ are globally hyperbolic;
\item Each $\Omega^\pm$ is maximal among connected open domains satisfying these two properties; 
\item The only open domains satisfying all three properties above are $\Omega^\pm$.
\end{itemize}
\end{theorem}
The notion of global hyperbolicity is central in 
General Relativity, see for example \cite{beem}. The global hyperbolicity
of $\Omega^\pm/\Gamma$ implies that it is 
homeomorphic to the product $( {\bf H}^2 / G) \times \RR$. It also implies that no element of $\Gamma$ preserves a null ray in
$\Omega^\pm$.  

Let $\Gamma$ be as in Theorem~\ref{thm:omegapd} and consider its action, for instance, on $\Omega^+$.  Since $\Omega^+/\Gamma$ is globally hyperbolic, it admits a {\em Cauchy hypersurface}, a spacelike  surface  $S_0$ which meets every complete causal curve and with complement consisting of two connected components.  The universal covering  $\widetilde{S}_0$ is
$\Gamma$-invariant.  
The subset $\Pho_0\subset \Pho$ comprising photons which
intersect $\widetilde{S}_0$ is open.

We claim that $\Gamma$ acts freely and properly on
$\Pho_0$.  Indeed, let $K\leq \Pho_0$ be a compact set.  Then $K$ is
contained in a product of compact subsets $K_1\times K_2$, where
$K_1\subset \widetilde{S}_0$ and $K_2\subset S^1$, the set of photon directions.
The action of $\Gamma$ restricts to a Riemannian action on $\widetilde{S}_0$.
Thus the set 
\begin{equation*}
\{ \gamma\in\Gamma~\mid~\gamma(K_1)\cap K_1\neq\emptyset\} 
\end{equation*}
 is finite.  As $\widetilde{S}_0$ is spacelike, it follows that
$\{ \gamma\in\Gamma~\mid~\gamma(K)\cap K\neq\emptyset\}$ is finite
too.  Finally, global hyperbolicity of $\Omega^+/\Gamma$ implies
that no photon intersecting $\Omega^+$ is
invariant under the action of any element of $\Gamma$.

\begin{corollary}
There exists a non-empty open subset of $\Pho$ on which
$\Gamma$ acts freely and properly discontinuously.
\end{corollary}

\subsection{Some questions}
So far we have considered groups of transformations of $\EinTO$ and $\Pho$ arising from discrete groups of Minkowski isometries.  Specifically, we have focused on groups generated by spine reflections associated to spacelike circles intersecting in a point.

{\em Question:} Describe the action on $\EinTO$ of a group generated by spine reflections corresponding to non-intersecting spacelike circles.  In particular, determine the possible dynamics of such an action.

A related question is:

{\em Question:} What does a crooked surface look like when its spine does not pass through $p_\infty$, or the lightcone at infinity altogether?  Describe the action of the associated group of spine reflections.

More generally, we may wish to consider other involutions in the automorphism group of a crooked surface.

{\em Question:} Describe the action on $\EinTO$ of a group generated by involutions, in terms of their associated crooked surfaces.

As for the action on photon space, here is a companion question to those asked in \S\ref{sec.dynamic}:

{\em Question:} Given a group generated by involutions, what is the maximal open subset of $\Pho$ on which the group acts properly discontinuously?

\printindex

\end{document}